\documentclass[psamsfonts,reqno]{amsart}
\usepackage{amssymb,eucal}
\usepackage{graphics}

\usepackage{hyperref}
\hypersetup{backref=true,  
colorlinks=true,
citecolor=green,
linkcolor=green,
anchorcolor=green,
} 

\theoremstyle{plain}

\newtheorem{lemma}{Lemma}[section]
\newtheorem{theorem}[lemma]{Theorem}
\newtheorem{proposition}[lemma]{Proposition}
\newtheorem{corollary}[lemma]{Corollary}

\theoremstyle{definition}
\newtheorem{definition}[lemma]{Definition}
\newtheorem{example}[lemma]{Example}

\theoremstyle{remark}
\newtheorem{remark}[lemma]{Remark}
\newtheorem{notation}[lemma]{Notation}
\newtheorem*{note}{Note}

\newtheorem*{stat}{\name}
\newcommand{\name}{testing}

\newenvironment{all}[1]{\renewcommand{\name}{#1}\begin{stat}}
                        {\end{stat}}

\newcommand{\qedc}{{\qed}~{\rm Claim~{\theclaim}.}}
\newcommand{\qedsc}{{\qed}~{\rm Claim.}}

\numberwithin{equation}{section}
\numberwithin{figure}{section}

\newcommand{\pup}[1]{\textup{(}{#1}\textup{)}}
\newcommand{\jirr}{join-ir\-re\-duc\-i\-ble}
\newcommand{\mirr}{meet-ir\-re\-duc\-i\-ble}

\newcommand{\jsd}{join-sem\-i\-dis\-trib\-u\-tive}
\newcommand{\jsdy}{join-sem\-i\-dis\-trib\-u\-tiv\-i\-ty}
\newcommand{\msd}{meet-sem\-i\-dis\-trib\-u\-tive}
\newcommand{\msdy}{meet-sem\-i\-dis\-trib\-u\-tiv\-i\-ty}

\newcommand{\Msdy}{Meet-sem\-i\-dis\-trib\-u\-tiv\-i\-ty}

\newcommand{\msubsemi}{meet-sub\-sem\-i\-lat\-tice}

\newcommand{\contr}{a contradiction}

\newcommand{\set}[1]{\{#1\}}
\newcommand{\setm}[2]{\set{#1\mid#2}}

\newcommand{\seq}[1]{\langle{#1}\rangle}

\newcommand{\famm}[2]{(#1\mid#2)}
\newcommand{\Famm}[2]{\bigl(#1\mid#2\bigr)}
\newcommand{\fF}{\mathfrak{F}}
\newcommand{\fS}{\mathfrak{S}}

\newcommand{\oo}[1]{\left]{#1}\right[}
\newcommand{\oc}[1]{\left]{#1}\right]}
\newcommand{\co}[1]{\left[{#1}\right[}
\newcommand{\cc}[1]{\left[{#1}\right]}
\newcommand{\so}[1]{\boldsymbol{\delta}_{#1}}

\DeclareMathOperator{\Inv}{inv}
\DeclareMathOperator{\Ker}{Ker}
\DeclareMathOperator{\Int}{int}
\DeclareMathOperator{\Cl}{cl}
\DeclareMathOperator{\J}{Ji}
\DeclareMathOperator{\M}{Mi}
\newcommand{\pji}[3]{\seq{{#1},{#2};{#3}}}
\newcommand{\pUji}[3]{\seq{{#1},{#2}}_{#3}}

\newcommand{\op}{\mathrm{op}}
\DeclareMathOperator{\rD}{D}

\newcommand{\cC}{\mathcal{C}}

\newcommand{\cF}{\mathcal{F}}

\newcommand{\cI}{\mathcal{I}}

\renewcommand{\SS}{\mathbb{S}}

\newcommand{\bD}{\mathbin{\boldsymbol{D}}}

\newcommand{\dnw}{\mathbin{\downarrow}}
\newcommand{\ddnw}{\mathbin{\downdownarrows}}
\newcommand{\upw}{\mathbin{\uparrow}}

\newcommand{\utr}{\trianglelefteq}

\newcommand{\Pow}{\mathfrak{P}}
\newcommand{\cpl}{\mathsf{c}}

\newcommand{\jz}{$(\vee,0)$}

\newcommand{\jzu}{$(\vee,0,1)$}
\newcommand{\jzs}{\jz-semi\-lat\-tice}

\newcommand{\jzus}{\jzu-semi\-lat\-tice}

\newcommand{\muh}{$(\wedge,1)$-ho\-mo\-mor\-phism}

\newcommand{\js}{join-sem\-i\-lat\-tice}

\newcommand{\jh}{join-ho\-mo\-mor\-phism}
\newcommand{\mh}{meet-ho\-mo\-mor\-phism}

\newcommand{\res}{\mathbin{\restriction}}
\newcommand{\es}{\varnothing}

\newcommand{\sA}{\mathbf{\mathsf{A}}}
\newcommand{\sB}{\mathsf{B}}

\newcommand{\sD}{\mathsf{D}}
\newcommand{\sM}{\mathsf{M}}
\newcommand{\sN}{\mathsf{N}}
\newcommand{\sP}{\mathsf{P}}

\newcommand{\sa}{\mathsf{a}}
\renewcommand{\sb}{\mathsf{b}}

\newcommand{\se}{\mathsf{e}}
\renewcommand{\sf}{\mathsf{f}}
\newcommand{\su}{\mathsf{u}}
\newcommand{\sv}{\mathsf{v}}

\newcommand{\sx}{\mathsf{x}}
\newcommand{\sy}{\mathsf{y}}
\newcommand{\hsx}{\hat{\mathsf{x}}}
\newcommand{\hsy}{\hat{\mathsf{y}}}
\newcommand{\sz}{\mathsf{z}}

\newcommand{\FL}{\operatorname{\mathrm{F}_{\mathbf{L}}}}

\newcommand{\ba}{\boldsymbol{a}}
\newcommand{\bb}{\boldsymbol{b}}
\newcommand{\hba}{\hat{\boldsymbol{a}}}
\newcommand{\hbb}{\hat{\boldsymbol{b}}}
\newcommand{\bc}{\boldsymbol{c}}
\newcommand{\be}{\boldsymbol{e}}
\renewcommand{\bf}{\boldsymbol{f}}
\newcommand{\bp}{\boldsymbol{p}}
\newcommand{\bq}{\boldsymbol{q}}

\newcommand{\bx}{\boldsymbol{x}}
\newcommand{\by}{\boldsymbol{y}}

\newcommand{\tb}{\tilde{\mathsf{b}}}

\newcommand{\Veg}[1]{\ensuremath{(\mathrm{Veg}_{#1})}}
\newcommand{\Gzp}[1]{\ensuremath{\mathrm{Gzp}(#1)}}

\begin{document}

\title[Associahedra and permutohedra]{Sublattices of associahedra and permutohedra}

\author[L. Santocanale]{Luigi Santocanale}
\address{Laboratoire d'Informatique Fondamentale de Marseille\\
Aix-Marseille Universit\'e\\
13288 Marseille Cedex 9 \\
France}
\email{luigi.santocanale@lif.univ-mrs.fr}
\urladdr{http://www.lif.univ-mrs.fr/\~{}lsantoca/}

\author[F. Wehrung]{Friedrich Wehrung}
\address{LMNO, CNRS UMR 6139\\
D\'epartement de Math\'ematiques\\
Universit\'e de Caen\\
14032 Caen Cedex\\
France}
\email{wehrung@math.unicaen.fr}
\urladdr{http://www.math.unicaen.fr/\~{}wehrung}



\date{\today}

\begin{abstract}
Gr\"atzer asked in 1971 for a characterization of sublattices of Tamari lattices. A natural candidate was coined by McKenzie in 1972 with the notion of a \emph{bounded homomorphic image of a  free lattice}---in short, \emph{bounded lattice}. Urquhart proved in 1978 that every Tamari lattice is bounded (thus so are its sublattices). Geyer conjectured in 1994 that every finite bounded lattice embeds into some Tamari lattice.

We disprove Geyer's conjecture, by introducing an infinite collection of lattice-theoretical identities that hold in every Tamari lattice, but not in every finite bounded lattice. Among those finite counterexamples, there are the \emph{permutohedron} on four letters~$\sP(4)$, and in fact two of its subdirectly irreducible retracts, which are \emph{Cambrian lattices of type~A}.

For natural numbers~$m$ and~$n$, we denote by~$\sB(m,n)$ the (bounded) lattice obtained by doubling a join of~$m$ atoms in an $(m+n)$-atom Boolean lattice. We prove that~$\sB(m,n)$ embeds into a Tamari lattice if{f} $\min\set{m,n}\leq1$, and that $\sB(m,n)$ embeds into a permutohedron if{f} $\min\set{m,n}\leq2$. In particular, $\sB(3,3)$ cannot be embedded into any permutohedron. Nevertheless we prove that~$\sB(3,3)$ is a homomorphic image of a sublattice of the permutohedron on~$12$ letters.
\end{abstract}

\maketitle

\section{Introduction}\label{S:Intro}

For every positive integer~$n$, the set~$\sA(n)$ of all binary
bracketings of $n+1$ symbols~$x_0$, $x_1$, \dots, $x_n$ can be
partially ordered by the reflexive, transitive closure of the binary
relation consisting of all the pairs $(s,t)$ where~$t$ is obtained
from~$s$ by replacing a subword of the form $(uv)v$ by $u(vw)$. The
study of the poset~$\sA(n)$ originates in Tamari~\cite{Tam62}, and is
then pursued in many papers. In particular, Friedman and
Tamari~\cite{FrTa67} prove that~$\sA(n)$ is a \emph{lattice}, that is,
every pair~$x$, $y$ of elements has a least upper bound (join) $x\vee
y$ and a greatest lower bound (meet) $x\wedge y$. The lattice~$\sA(n)$
is called a \emph{Tamari lattice}, or \emph{associativity lattice}, in
Bennett and Birkhoff~\cite{BeBi94}. The elements of~$\sA(n)$ are in one-to-one correspondence with the vertices of the \emph{Stasheff polytope}, also called \emph{associahedron} (cf. Stasheff~\cite{Sta63}).

Gr\"atzer asked in Problem~6 of~\cite{FCLat} (see also Problem~I.1 of
Gr\"atzer~\cite{GLT2}) for a characterization of all sublattices of
Tamari lattices. Soon after, McKenzie~\cite{McKe72}
introduced a lattice-theoretical property that later proved itself
fundamental, namely being a \emph{bounded homomorphic image of a free
  lattice} (see Section~\ref{S:NotaTerm} for precise
definitions). Since then the convention of calling such lattices
\emph{bounded lattices} (not to be confused with lattices with a least
and a largest element) has established itself. Among the two simplest
nondistributive lattices~$\sM_3$ and~$\sN_5$
(cf. Figure~\ref{Fig:M3N5}), $\sN_5$ (on the right hand side of the
picture) is bounded while~$\sM_3$ (on the left hand side) is not.

\begin{figure}[htb]
\includegraphics{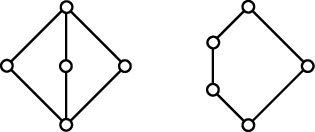}
\caption{The lattices $\sM_3$ and $\sN_5$}
\label{Fig:M3N5}
\end{figure}

Urquhart proved in \cite[Corollary, page~55]{Urqu78} that every Tamari lattice is bounded. Since every sublattice of a finite bounded lattice is bounded, it follows that~$\sM_3$ cannot be embedded into any Tamari lattice. On the other hand~$\sN_5$ is itself a Tamari lattice (namely~$\sA(3)$), and every distributive lattice with~$n$ \jirr\ elements can be embedded into~$\sA(n+1)$ (cf. Markowsky~\cite[page~288]{Mark92}). This led to a plausible conjecture as to which lattices can be embedded into some Tamari lattice, namely:
\begin{center}\em
Can every finite bounded lattice be embedded into some Tamari lattice?
\end{center}
This conjecture was first stated in Geyer~\cite[page~106]{Geyer94}.

Finite bounded lattices are exactly those that can be obtained, starting with the one-element lattice, by applying a finite sequence of instances of the so-called \emph{doubling construction} on closed intervals (see Freese, Je\v{z}ek, and Nation~\cite[Corollary~2.44]{FJN}). At the bottom of the hierarchy of bounded lattices, we can find those obtained by doubling a point (viewed as a one-element interval) in a finite Boolean lattice. Denote by~$\sB(m,n)$ the lattice obtained by doubling the join of~$m$ atoms in an $(m+n)$-atom Boolean lattice (cf. Section~\ref{S:Veg2}). We prove in Corollary~\ref{T:EmbBn01} that~$\sB(m,n)$ embeds into some Tamari lattice if{f} $\min\set{m,n}\leq1$. This settles Geyer's conjecture in the negative.

Our proof involves the construction of an infinite collection of
lattice-theoretical identities, the \emph{Gazpacho identities}
(Section~\ref{S:Gazpacho}). We prove that every Tamari lattice
satisfies all Gazpacho identities (Theorem~\ref{T:GzpinAssoc}). The
simplest Gazpacho identity, \Gzp{1,1}, is renamed~\Veg{1} in
Section~\ref{S:Veg1}, and we find there our first example of a finite
bounded lattice that does not satisfy some Gazpacho identity
(namely~\Veg{1}). This lattice, denoted by~$\sA_{\set{3}}(4)$ and
represented on the right hand side of Figure~\ref{Fig:Cambrian}, is a
retract of the \emph{permutohedron}~$\sP(4)$. (As usual, the permutohedron~$\sP(n)$ on~$n$ letters is defined as the symmetric group of order~$n$ endowed with the weak Bruhat order.) Thus, we infer that the permutohedron~$\sP(4)$ has no lattice embedding into any
Tamari lattice: it does not satisfy the identity~\Veg{1} satisfied by
every Tamari lattice. More generally, we introduce a family of
lattices $\sA_{U}(n)$, for $U \subseteq \set{1,\ldots ,n}$, that are
\emph{retracts}---with respect to the lattice operations---of the
permutohedron~$\sP(n)$, cf. Proposition~\ref{P:PUnjoinfitsPn}. We
verify with Proposition~\ref{P:CambCong} the identity between our
lattices~$\sA_U(n)$ and Reading's \emph{Cambrian lattices of type~A}
\cite{Read06}.  In particular, we characterize in
Corollary~\ref{C:CambCong} the Cambrian lattices of type~A as the
quotients of permutohedra by their \emph{minimal \mirr\ congruences}.

As seen above, another source of finite bounded lattices that cannot be embedded into any Tamari lattice is provided by the lattices~$\sB(m,n)$, for $\min\set{m,n}\geq2$. We introduce in Section~\ref{S:Veg2} a weakening, denoted by~\Veg{2}, of~\Gzp{2,2}, that is not satisfied by~$\sB(2,2)$ (Corollary~\ref{C:B22Config}). Hence~$\sB(2,2)$ is another counterexample to Geyer's conjecture. This lattice is represented in the right hand side of Figure~\ref{Fig:B13B22}.

Our negative embedding result for the permutohedron~$\sP(4)$ raises
the analogue of Geyer's question for permutohedra: namely, \emph{can
  every finite bounded lattice be embedded into some permutohedron}?
Again, it is known that every permutohedron is bounded
(cf. Caspard~\cite{Casp00}). Since every Tamari lattice~$\sA(n)$ is a
sublattice (and, in fact, a \emph{retract}, see
Corollary~\ref{C:PUnjoinfitsPn}) of the corresponding
permutohedron~$\sP(n)$, every sublattice of a Tamari lattice is also a
sublattice of a permutohedron. We disprove the question above in
Theorem~\ref{T:B33}, by proving that \emph{the lattice~$\sB(3,3)$
  cannot be embedded into any permutohedron}. Our proof starts with
the observation that since~$\sB(3,3)$ is subdirectly irreducible, if
it embeds into some permutohedron~$\sP(\ell)$, then it embeds into
some Cambrian lattice~$\sA_U(\ell)$.

Unlike our negative solution of Geyer's conjecture, which involves an
\emph{identity} that holds in all associahedra but not in~$\sB(2,2)$,
our negative embedding result for~$\sB(3,3)$ does not produce an
identity. There is a good reason for this. Namely, $\sB(3,3)$ is,
using terminology from McKenzie~\cite{McKe72}, \emph{splitting} (which
means finite, bounded, and subdirectly irreducible), hence there is a
lattice-theoretical identity that holds in a lattice~$L$
if{f}~$\sB(3,3)$ does not belong to the lattice variety generated
by~$L$. Such an identity is constructed, using known algorithms,
in~\eqref{Eq:Spl1B330}. Then, with the assistance of the software
\texttt{Prover9 - Mace4}, we prove that the Cambrian
lattice~$\sA_U(12)$, for $U=\set{5,6,9,10,11}$, does not
satisfy that identity. In particular, this shows that
although~$\sB(3,3)$ satisfies all the identities satisfied by all
permutohedra (and even all the identities satisfied by~$\sP(12)$), it
cannot be embedded into any permutohedron. Hence, our negative
embedding result for~$\sB(3,3)$ (Theorem~\ref{T:B33}) cannot be proved
\emph{via} a separating identity.

\subsection*{A small discussion about terminology}

In the same manner the lattices~$\sA(n)$ are usually called ``Tamari
lattices'', it would seem natural to call the lattices~$\sP(n)$
``Guilbaud and Rosenstiehl lattices'', after Guilbaud and
Rosenstiehl~\cite{GuRo63} (cf. Section~\ref{S:BasicPerm}). Tradition
decided otherwise, and the lattice~$\sP(n)$ is often\footnote{The lattice~$\sP(n)$ is also often called the ``symmetric group of order~$n$ with the weak Bruhat order''. We will not use that terminology.} called the ``permutohedron on~$n$ letters'' (or,
sometimes, ``permutohedron lattice on~$n$ letters''). We should point
out that the term ``permutohedron'' often denotes either a polytope
(the convex hull of permutation matrices) or a graph (the adjacency
graph of the polytope); the traditional naming for the permutohedron lattice stems from the fact that its undirected covering graph coincides with the adjacency graph of the polytope. 
However our present work is lattice-theoretical and thus we shall use the term ``permutohedron'' only in the lattice-theoretical sense.

Now, according to the same logic, it would have made sense to call ``associahedron'' the lattice~$\sA(n)$. As for permutohedra, this term usually denotes either a polytope or a graph, the latter being the undirected covering graph of the
lattice~$\sA(n)$. However, it follows from work by Reading~\cite{Read06,Read07a} (mainly Theorem~1.3 in the first paper
and Theorem~4.1 in the second paper) that many other lattices share the same undirected covering graph; these are the \emph{Cambrian lattices of type~A}, denoted in the present paper by~$\sA_U(n)$ (cf. Sections~\ref{S:Alterno} and~\ref{S:SubdDec}). In particular, each of those lattices would also deserve to be called
``associahedron''. Because of that possible ambiguity, we shall keep
calling the~$\sA(n)$ ``Tamari lattices''.

\section{Basic notation and terminology}
\label{S:NotaTerm}

We set
\begin{align*}
  [n]&=\set{1,\dots,n}\,,\\
  \cI_n&=\setm{(i,j)\in[n]\times[n]}{i<j}\,, \\
  \Delta_n&=\setm{(i,i)}{i\in[n]}\,,
\end{align*}
for every natural number~$n$.

For a subset $X$ in a poset $P$, we set
 \begin{align*}
 P\dnw X&=\setm{p\in P}{(\exists x\in X)(p\leq x)}\,,\\
 P\ddnw X&=\setm{p\in P}{(\exists x\in X)(p<x)}\,,\\
 P\upw X&=\setm{p\in P}{(\exists x\in X)(p\geq x)}\,;
 \end{align*}
furthermore, we set $P\dnw x=P\dnw\set{x}$, $P\ddnw x=P\ddnw\set{x}$, and $P\upw x=P\upw\set{x}$, for each $x\in X$. For subsets~$X$ and~$Y$ of~$P$, we say that~$X$ \emph{refines}~$Y$, in notation $X\ll Y$, if $X\subseteq P\dnw Y$. For elements $a,b\in P$, we set
 \begin{align*}
 \cc{a,b}&=\setm{p\in P}{a\leq p\leq b}\,,\\
 \co{a,b}&=\setm{p\in P}{a\leq p<b}\,,\\
 \oc{a,b}&=\setm{p\in P}{a<p\leq b}\,,\\
 \oo{a,b}&=\setm{p\in P}{a<p<b}\,. 
 \end{align*}
Here we stray away from the usual convention of denoting intervals in the form~$[a,b)$ or~$(a,b]$ for half-open intervals and~$(a,b)$ for open intervals. The reason for this is that the present paper involves the notations~$(a,b)$ (for pairs of elements), $\oo{a,b}$ (for open intervals), and~$\seq{a,b}$ (for \jirr\ elements in associahedra).

We shall denote by~$P^\op$ the poset with the same underlying set as~$P$ but ordering reversed.

A lattice~$L$ is \emph{\jsd} if $x\vee y=x\vee z$ implies that $x\vee
y=x\vee(y\wedge z)$, for all $x,y,z\in L$. \emph{\Msdy} is defined
dually, and \emph{semidistributivity} is the conjunction of \jsdy\ and
\msdy. A \emph{lattice term} is obtained from variables by repeatedly
composing the meet and the join operations, so for example
$(\sx\wedge\sy)\vee(\sx\vee\sz)$ is a lattice term (we shall use lower
case Sans Serif fonts, such as~$\sx$, $\sy$, $\sz$, $\su$, $\sv$\dots,
for either variables or terms). A (lattice-theoretical) \emph{identity} is a statement of the
form $\su=\sv$ (or $\su\leq\sv$, equivalent to $\su=\su\wedge\sv$) for
lattice terms~$\su$ and~$\sv$. A lattice~$L$ \emph{satisfies the identity~$\su=\sv$} if $\su(\vec{a})=\sv(\vec{a})$ for each
assignment~$\vec{a}$ from the variables of either~$\su$ or~$\sv$ to
the elements of~$L$. A \emph{variety of lattices} is the class of all
lattices that satisfy a given set of identities.
 
A nonzero element~$p$ in~$L$ is \emph{\jirr} if $p=\bigvee X$ implies that $p\in X$ for each finite nonempty subset~$X$ of~$L$. Meet-irreducible elements are defined dually. We denote by~$\J(L)$ (resp., $\M(L)$) the set of all \jirr\ (resp., \mirr) elements of~$L$. A \emph{lower cover} of an element $p\in L$ is an element $x<p$ in~$L$ such that $\oo{x,p}=\es$. Upper covers are defined dually. We denote by~$p_*$ (resp., $p^*$) the lower cover (resp., upper cover) of~$p$ in case it exists and it is unique. For a finite lattice~$L$, $\J(L)$ is exactly the set of all the elements of~$L$ that have a unique lower cover; and dually for~$\M(L)$. In that case, we define binary relations~$\nearrow$ and~$\searrow$ on~$L$ by setting
 \begin{align*}
 x\nearrow y&\ \Longleftrightarrow\ 
 (y\in\M(L)\text{ and }x\nleq y\text{ and }x\leq y^*)\,,\\
 y\searrow x&\ \Longleftrightarrow\ 
 (x\in\J(L)\text{ and }x\nleq y\text{ and }x_*\leq y)\,,
 \end{align*}
for all $x,y\in L$. Then~$L$ is \msd\ if{f} for each $p\in\J(L)$, there exists a largest element $u\in L$ such that $u\searrow p$; this element is then denoted by $\kappa_L(p)$, or $\kappa(p)$ in case~$L$ is understood, and it is \mirr\ (cf. Freese, Je\v{z}ek, and Nation \cite[Theorem~2.56]{FJN}). A similar statement holds for \jsdy\ and $\kappa^\op(u)$, for $u\in\M(L)$, instead of \msdy\ and $\kappa(p)$, for $p\in\J(L)$.

The \emph{join-dependency relation} is the binary relation~$\bD_L$ on~$L$ defined by $a\bD_Lq$, or $a\bD q$ in case~$L$ is understood, by
 \begin{equation}\label{Eq:DefJoinDep}
 a\bD_Lq\ \Leftrightarrow\ \bigl(q\in\J(L)\text{ and }a\neq q\text{ and }
 (\exists x\in L) (a\leq q\vee x\text{ and }a\nleq q_*\vee x)\bigr)\,,
 \end{equation}
for all $a,q\in L$. A \emph{join-cover} of $a \in L$ is a finite subset~$C \subseteq L$ such that $a \leq \bigvee C$.  A join-cover~$C$ of~$a$ is \emph{nontrivial} if $a\notin L\dnw C$.  A join-cover~$C$ of~$a$ is \emph{minimal} if, for every join-cover~$D$ of~$a$, $D \ll C$ implies $C \subseteq D$. It is well-known \cite[Lemma~2.31]{FJN} that, if~$L$ is a finite lattice and $a,q \in L$,
 \begin{equation*}
 a\bD_Lq\ \Leftrightarrow\ \text{there exists a minimal nontrivial
 join-cover $C$ of $a$ such that $q \in C$}.
 \end{equation*}
A surjective homomorphism $f\colon K\twoheadrightarrow L$ is
\emph{bounded} if $f^{-1}\set{x}$ has a least and a largest element,
for each $x\in L$. McKenzie recognized in~\cite{McKe72} the
fundamental role played by lattices which are \emph{bounded homomorphic images of free lattices}. Since then, those lattices have been mostly called \emph{bounded  lattices}. Every bounded lattice is semidistributive (apply \cite[Theorem~2.20]{FJN} and its dual), but the converse fails, even for finite lattices (see the example represented in
\cite[Figure~5.5]{FJN}).
Bounded lattices are called \emph{congruence-uniform} in
Reading~\cite{Read04}, unfortunately the latter terminology is also in
use for lattices in which all congruence classes, with respect to any
given congruence, have the same cardinality, so we shall use here the widely established ``bounded'' terminology here.

A finite lattice~$L$ is bounded if{f} the join-dependency relation is cycle-free on the \jirr\ elements of both~$L$ and~$L^\op$ (cf. \cite[Corollary~2.39]{FJN}). The finite bounded lattices are exactly those that can be obtained by starting from the one-element lattice and then applying a finite succession of the so-called \emph{doubling operation} on closed intervals, cf. Freese, Je\v{z}ek, and Nation \cite[Theorem~2.44]{FJN}.

As shown by the following result from Freese, Je\v{z}ek, and Nation \cite[Lemma~11.10]{FJN}, the relation~$\bD_L$ can be easily obtained from the arrow relations between~$\J(L)$ and~$\M(L)$.

\begin{lemma}\label{L:Arr2D}
Let $p$, $q$ be distinct \jirr\ elements in a finite lattice~$L$. Then $p\bD_Lq$ if{f} there exists $u\in\M(L)$ such that $p\nearrow u\searrow q$.
\end{lemma}

\section{Basic concepts about permutohedra}\label{S:BasicPerm}

Throughout this section we shall define permutohedra in a way suited to our needs (Definition~\ref{D:DefOrdPerm}) and relate that definition to those of some earlier works. We fix a natural number~$n$.

A subset~$\bx$ of~$\cI_n$ is \emph{closed} if it is transitive (viewed as a binary relation): that is, $(i,j)\in\bx$ and $(j,k)\in\bx$ implies that $(i,k)\in\bx$, for all $i,j,k\in[n]$. A subset~$\bx$ of~$\cI_n$ is \emph{open} (resp., \emph{clopen}), if $\cI_n\setminus\bx$ is closed (resp., both~$\bx$ and $\cI_n\setminus\bx$ are closed).

\begin{definition}\label{D:DefOrdPerm}
The \emph{permutohedron of index~$n$}, denoted by $\sP(n)$, is the set of all clopen subsets of~$\cI_n$, partially ordered by inclusion.
\end{definition}

The permutohedron was first defined in terms of the group~$\fS_n$ of
all permutations of~$[n]$, for each positive integer~$n$. We set
$\Inv(\sigma)=\setm{(i,j)\in\cI_n}{\sigma^{-1}(i)>\sigma^{-1}(j)}$ for
each $\sigma\in\fS_n$, the \emph{set of inversions} of~$\sigma$.  The
following result can be traced back to Guilbaud and Rosenstiehl
\cite[Th\'eor\`eme~2]{GuRo63}; see also Exercise~16, page~225 in
Bourbaki~\cite{LieBourb} (where it is established in the more general
context of finite Coxeter groups), Yanagimoto and Okamoto
\cite[Proposition~2.2]{YaOk69}.

\begin{lemma}\label{L:InvSepPerm}
The assignment $\sigma\mapsto\Inv(\sigma)$ defines a bijection from~$\fS_n$ onto the set of all clopen subsets of~$\cI_n$, for every positive integer~$n$.
\end{lemma}

It follows from Lemma~\ref{L:InvSepPerm} that one can define a partial ordering on~$\fS_n$ by setting
 \[
 \sigma\leq\tau\ \Leftrightarrow\ \Inv(\sigma)\subseteq\Inv(\tau)\,,
 \text{ for all }\sigma,\tau\in\fS_n\,,
 \]
 and this partial ordering is isomorphic to the permutohedron~$\sP(n)$
 (cf. Definition~\ref{D:DefOrdPerm}). The partial ordering defined above on~$\sP(n)$ turns out to be the well-known \emph{weak Bruhat ordering} on the symmetric group, see for example Bennett and Birkhoff~\cite[Section~5]{BeBi94}.

 The description of
 permutations \emph{via} clopen sets of inversions is a particular
 case of a more general construction, namely the description of the
 regions of a hyperplane arrangement \emph{via} bi-closed sets of
 hyperplanes. For details, we refer the reader to Bj\"orner, Edelman,
 and Ziegler~\cite{BEZ}, in particular in the Example at the bottom of
 page~269, also in Theorem~5.5, of that paper.

Since every intersection of closed sets is closed, every union of open sets is open. For a subset~$\bx$ of~$\cI_n$, we shall denote by $\Int(\bx)$ (resp., $\Cl(\bx)$) the largest open subset of~$\bx$ (resp., the least closed set containing~$\bx$). Hence~$\Cl(\bx)$ is the transitive closure of~$\bx$, while
 \begin{multline}\label{Eq:ExprInt(X)}
 \Int(\bx)=\bigl\{(i,j)\in\cI_n\mid (\forall m>0)
 (\forall i=s_0<s_1<\cdots<s_m=j)\\
 (\exists l<m)\bigl((s_l,s_{l+1})\in\bx\bigr)\bigr\}\,.
 \end{multline}
 The following lemma is crucial in establishing
 Proposition~\ref{P:Perm(n)Latt}. It is implicit in the proof of
 Guilbaud and Rosenstiehl \cite[Section~VI.A]{GuRo63}, Yanagimoto and
 Okamoto \cite[Theorem~2.1]{YaOk69}, and it is stated explicitly in
 Santocanale \cite[Lemma~2.6]{Sant07}.

\begin{lemma}\label{Int(closed)}
The set $\Cl(\bx)$ is open, for each open $\bx\subseteq\cI_n$. Dually, the set $\Int(\bx)$ is closed, for each closed $\bx\subseteq\cI_n$.
\end{lemma}

{}From Lemma~\ref{Int(closed)} it follows that for all $\bx,\by\in\sP(n)$, there exists a largest element of~$\sP(n)$ contained in~$\bx\cap\by$, namely $\Int(\bx\cap\by)$. Dually, there exists a least element of~$\sP(n)$ that contains $\bx\cup\by$, namely $\Cl(\bx\cup\by)$. Therefore, we get the following result, first established in Guilbaud and Rosenstiehl \cite[Section~VI.A]{GuRo63}, see also Yanagimoto and Okamoto \cite[Theorem~2.1]{YaOk69}.

\begin{proposition}\label{P:Perm(n)Latt}
The poset~$\sP(n)$ is a lattice. The meet and the join in~$\sP(n)$ are given by
 \[
 \bx\wedge\by=\Int(\bx\cap\by)\,,\quad\bx\vee\by=\Cl(\bx\cup\by)\,,
 \]
for all $\bx,\by\in\sP(n)$.
\end{proposition}

Hence, the permutohedron~$\sP(n)$ it is often called the \emph{lattice of all permutations of~$n$ letters}.

\begin{proposition}\label{P:Perm(n)compl}
The lattice~$\sP(n)$ is complemented. Moreover, the assignment $\bx\mapsto\bx^\cpl=\cI_n\setminus\bx$ defines an involutive dual automorphism of~$\sP(n)$ that sends each clopen subset~$\bx$ to a lattice-theoretical complement of~$\bx$ in~$\sP(n)$.
\end{proposition}

The least element of $\sP(n)$ is~$\es$, it is the set of inversions of the identity permutation. The largest element of~$\sP(n)$ is~$\cI_n$; it is the set of inversions of the permutation $i\mapsto n+1-i$.

\section{Join-irreducible elements in the permutohedron}\label{S:JirrPerm}

Throughout this section we shall fix a natural number~$n$. We shall describe the join- and \mirr\ elements of~$\sP(n)$, state a few lemmas needed for further sections, and indicate how they imply Caspard's result that all permutohedra are bounded.

\begin{notation}\label{Not:DabU}
We set $\cF_n=\setm{(a,b,U)}{(a,b)\in\cI_n\,,\ U\subseteq[a,b]\,,\ a\notin U\,,\text{ and }b\in U}$, and, for each $(a,b,U)\in\cF_n$, we set
 \[
 \pji{a}{b}{U}=\cI_n\cap\bigl(([a,b]\setminus U)\times U\bigr)\,.
 \]
\end{notation}

The following description of the \jirr\ elements in the lattice~$\sP(n)$ is contained in Santocanale \cite[Section~4]{Sant07}, see in particular Example~4.10 of that paper. By using Proposition~\ref{P:Perm(n)compl}, the description of \mirr\ elements follows.

\begin{lemma}\label{L:DescJ(Permn)}
  The \jirr\ \pup{resp., \mirr} elements of $\sP(n)$ are exactly those
  of the form $\pji{a}{b}{U}$ \pup{resp., $\pji{a}{b}{U}^\cpl$}, for
  $(a,b,U)\in\cF_n$.
\end{lemma}

\begin{lemma}\label{L:jirrSingl}
The equality $\pji{a}{b}{U}_*=\pji{a}{b}{U}\setminus\set{(a,b)}$ holds, for each triple $(a,b,U)\in\cF_n$.
\end{lemma}

Characterizations of the table of~$\sP(n)$, that is, the order between
\jirr\ elements and \mirr\ elements, and of the relations~$\searrow$
and~$\nearrow$ appear in Duquenne and Cherfouh \cite[Lemma~9]{DuCh94}
and Caspard \cite[Proposition~2]{Casp99}, respectively.  The previous
description of the \jirr\ elements by triples from $\cF_n$ yields the
following lemma in a straightforward way.

\begin{lemma}\label{L:ArrPrecJirr}
  Let $(a,b,U)\in\cF_n$. Set $\widetilde{U}=(\oc{a,b}\setminus
  U)\cup\set{b}$. Then $\bx\searrow\pji{a}{b}{U}$ if{f}
  $\bx\leq\pji{a}{b}{\widetilde{U}}^\cpl$, for each $\bx\in\sP(n)$.
\end{lemma}

Consequently, we obtain that for each $(c,d,V)\in\cF_n$,
$\pji{c}{d}{V}^\cpl$ lies above $\pji{a}{b}{U}_*$ but not above
$\pji{a}{b}{U}$ if{f} $\pji{a}{b}{\widetilde{U}}\leq\pji{c}{d}{V}$,
that is, $c\leq a<b\leq d$ and $\widetilde{U}=V\cap[a,b]$. It follows
that $\kappa_{\sP(n)}(\pji{a}{b}{U})=\pji{a}{b}{\widetilde{U}}^\cpl$.
By using \cite[Theorem~2.56]{FJN}, it follows that~$\sP(n)$ is
meet-semidistributive. Since~$\sP(n)$ is self-dual, we obtain that it is
semidistributive. This result was first obtained simultaneously by
Duquenne and Cherfouh \cite[Theorem~3]{DuCh94} and Le Conte de
Poly-Barbut \cite[Lemme~9]{LeCPB94} (in the latter paper the result
was extended to all \emph{Coxeter lattices}).

We set
 \begin{equation}\label{Eq:RestrSubset}
 U\res[i,j]=(U\cap\oc{i,j})\cup\set{j}\,,\quad\text{for all }U\subseteq[n]
 \text{ and all }(i,j)\in\cI_n\,.
 \end{equation}
By using Lemma~\ref{L:Arr2D} together with Lemmas~\ref{P:Perm(n)compl} and~\ref{L:ArrPrecJirr}, we obtain the following characterization of the join-dependency relation on~$\sP(n)$. This characterization was obtained in Santocanale \cite[Example~4.10]{Sant07}.

\begin{proposition}\label{P:DrelPerm}
  Let $(a,b,U),(c,d,V)\in\cF_n$. Then the relation
  $\pji{a}{b}{U}\bD\pji{c}{d}{V}$ holds in $\sP(n)$ if{f}
  $[c,d]\subsetneqq[a,b]$ and $V=U\res[c,d]$.
\end{proposition}

This implies trivially that the join-dependency relation on $\sP(n)$ is a strict ordering on~$\J(\sP(n))$. In particular, this relation has no cycle. Since~$\sP(n)$ is self-dual (cf. Lemma~\ref{P:Perm(n)compl}), we obtain the following result from Caspard~\cite[Theorem~1]{Casp00}.

\begin{theorem}\label{T:Nath00}
The lattice~$\sP(n)$ is bounded.
\end{theorem}

It is noteworthy to observe the following characterization of minimal join-covers in~$\sP(n)$, which can be obtained as a consequence of Proposition~\ref{P:DrelPerm}. Although we will make no direct use of Proposition~\ref{P:MinJoinCov}, the authors of the present paper found this result useful in coining the relevant notion of a \emph{$U$-polarized measure} introduced in Definition~\ref{D:PolMeas}.

\begin{proposition}\label{P:MinJoinCov}
For $(a,b,U)\in\cF_n$, the minimal join-covers $C$ of $\pji{a}{b}{U}$ are exactly those of the form
  \[
  C=\setm{\pji{z_{i}}{z_{i +1}}{U\res[z_{i},z_{i+1}]}}{i<k}\,,
  \]
where~$k$ is a positive integer and $a=z_0<z_1<\cdots<z_k=b$.
\end{proposition}

\section{The lattices~$\sA_U(n)$ and Tamari lattices}\label{S:Alterno}

In this section we shall introduce the lattices~$\sA_U(n)$, that will turn out later to be the \emph{Cambrian lattices of type~A} (cf. Proposition~\ref{P:CambCong}), and the Tamari lattices~$\sA(n)$ as particular cases. We shall relate our definition of~$\sA(n)$ with the one used by Huang and Tamari~\cite{HuTa72}, and verify that there are arbitrarily large $3$-generated sublattices of Tamari lattices (Proposition~\ref{P:AssocNotLocFin}). We shall also verify that every lattice~$\sA_U(n)$ is a sublattice of the corresponding~$\sP(n)$ (Corollary~\ref{C:Int(Uclosed)}).

Observe from Definition~\ref{D:DefOrdPerm} that for a positive integer~$n$, the permutohedron~$\sP(n)$ consists of all the transitive subsets~$\ba\subseteq\cI_n$ such that
 \[
 (x,z)\in\ba\text{ implies that either }(x,y)\in\ba
 \text{ or }(y,z)\in\ba\,,\quad\text{for all }x<y<z\text{ in }[n]\,.
 \]
When the choice whether $(x,y)\in\ba$ or $(y,z)\in\ba$ is determined
by a subset $U$ of~$[n]$, we obtain the structures~$\sA_U(n)$ defined
below. Namely, let us denote by~$\sD_U(n)$ the collection of all subsets~$\ba$ of~$\cI_n$ such that $1\leq i<j<k\leq n$ and $(i,k)\in\ba$ implies that $(i,j)\in\ba$ in case~$j\in U$ and $(j,k)\in\ba$ in case $j\notin U$.
Observe that, in order to define $\sD_U(n)$, we need only to know the
interior $U \setminus \set{1,n}$, so
$\sD_U(n)=\sD_{U\setminus\set{1,n}}(n)$. 
\begin{definition}\label{D:Alterno}
We define $\sA_U(n)$ as the collection of all \emph{transitive} members of~$\sD_U(n)$, and we order~$\sA_U(n)$ by set-theoretical inclusion. For $U=[n]$, we set $\sA(n)=\sA_{[n]}(n)$, the \emph{Tamari lattice on~$n$}.
\end{definition}

We first explain the terminology ``Tamari lattice'', for our structure~$\sA(n)$, as follows. Denote by~$\sA'(n)$ the set of all
maps $f\colon[n]\to[n]$ such that
\begin{itemize}
\item $i\leq f(i)$, for each $i\in[n]$;

\item $i\leq j\leq f(i)$ implies that $f(j)\leq f(i)$, for all $i,j\in[n]$.
\end{itemize}
We endow~$\sA'(n)$ with the componentwise ordering.

Huang and Tamari proved in~\cite{HuTa72} that~$\sA'(n)$ is isomorphic to the originally defined Tamari lattice, defined as the poset of all bracketings of~$n+1$ symbols given an ordering defined from certain natural rewriting rules (see the Introduction). Thus we will be entitled to call~$\sA(n)$ a Tamari lattice once we establish the following easy result.

\begin{proposition}\label{P:A(n)=A'(n)}
The posets~$\sA(n)$ and~$\sA'(n)$ are isomorphic, for every positive integer~$n$.
\end{proposition}

\begin{proof}
We define maps $\varphi\colon\sA(n)\to\sA'(n)$ and $\psi\colon\sA'(n)\to\sA(n)$ as follows:

\begin{itemize}
\item For each $\bx\in\sA(n)$, $\varphi(\bx)$ is the map from~$[n]$ to~$[n]$ defined by
 \[
 \varphi(\bx)(i)=\text{largest }j\in[i,n]\text{ such that }
 \set{i}\times\oc{i,j}\subseteq\bx\,,
 \qquad\text{for each }i\in[n]\,.
 \]

\item For each $f\in\sA'(n)$, we define $\psi(f)=\setm{(i,j)\in\cI_n}{j\leq f(i)}$.
\end{itemize}
It is a straightforward exercise to verify that these assignments define mutually inverse, order-preserving maps between~$\sA(n)$ and~$\sA'(n)$.
\end{proof}

For a positive integer~$n$, we define elements $\ba_n,\bb_n,\bc_n\in\sA(n)$ by
 \begin{align*}
 \ba_n&=\seq{1,n}\,,\\
 \bb_n&=\bigcup\famm{\seq{i,i+1}}{i\text{ even}\,,\ 1\leq i<n}\,,\\
 \bc_n&=\bigcup\famm{\seq{i,i+1}}{i\text{ odd}\,,\ 1\leq i<n}\,.
 \end{align*}
It follows from the proof of Santocanale \cite[Proposition~5.16]{Sant07} that the cardinality of the sublattice of~$\sA(n)$ generated by $\set{\ba_n,\bb_n,\bc_n}$ goes to infinity as~$n$ goes to infinity. Therefore,

\begin{proposition}\label{P:AssocNotLocFin}
There are arbitrarily large $3$-generated sublattices of Tamari lattices.
\end{proposition}

In universal algebraic terms, Proposition~\ref{P:AssocNotLocFin} implies immediately that \emph{the variety of lattices generated by all Tamari lattices is not locally finite}.

Although we found the description of Tamari lattices by either~$\sA(n)$ or~$\sA'(n)$ more convenient for our present purposes, this is not the case for all applications. For example, bracket reversing in the original description of the Tamari lattice easily implies the well-known fact that~$\sA(n)$ is \emph{self-dual}. This self-duality is not apparent in either description of the Tamari lattice by~$\sA(n)$ or~$\sA'(n)$. It is implicit in Lemmas~8 and~9 of Urquhart~\cite{Urqu78}, and stated in Bennett and Birkhoff~\cite[page~139]{BeBi94}.

The corresponding dual automorphism of~$\sA'(n)$ can be described explicitly as follows. Extend every element $f\in\sA'(n)$ at the point~$0$ by setting
$f(0)=n$. Observe that the conditions~(i) and~(ii) defining~$\sA'(n)$ are still satisfied on $[0,n]$. Next, for each $f\in\sA'(n)$, define $\tilde{f}\colon[0,n]\to[0,n]$ by setting $\tilde{f}(0)=n$, and
 \[
 \tilde{f}(i)=\text{least }j\in[i,n]\text{ such that }n-i<f(n-j)\,,\quad
 \text{for each }i\in[n]\,.
 \]
The proof of the following result is then an easy exercise.

\begin{proposition}\label{P:gg*antiaut}
The assignment $f\mapsto\tilde{f}$ defines an involutive dual automorphism of~$\sA'(n)$.
\end{proposition}

We come now to the structures~$\sA_U(n)$.
Clearly, $\sD_U(n)$ is a sublattice of the powerset lattice of~$\cI_n$; in particular, it is distributive. Furthermore, $\sA_U(n)$ is a \msubsemi\ of~$\sD_U(n)$ containing the largest element (namely~$\cI_n$); in particular, it is a lattice.

A key point in understanding the lattice structure of~$\sA_U(n)$ is the following analogue of Lemma~\ref{Int(closed)}.

\begin{lemma}\label{L:Int(Uclosed)}
  The closure~$\Cl(\bx)$ belongs to $\sA_U(n)$, for each
  $\bx\in\sD_U(n)$. Consequently, $\Cl(\bx)$ is the least element
  of~$\sA_U(n)$ containing~$\bx$.
\end{lemma}

\begin{proof}
Let $i<j<k$ with $(i,k)\in\Cl(\bx)$. By definition, there are a positive integer~$m$ and $i=s_0<s_1<\cdots<s_m=k$ such that $(s_u,s_{u+1})\in\bx$ for each $u<m$. Let $l<m$ such that $s_l<j\leq s_{l+1}$. If $j=s_{l+1}$, then the chain $i=s_0<s_1<\cdots<s_{l+1}=j$ witnesses the relation $(i,j)\in\Cl(\bx)$. Now suppose that $j<s_{l+1}$. If $j\in U$, then $(s_l,j)\in\bx$ and the chain $i=s_0<s_1<\cdots<s_l<j$ witnesses the relation $(i,j)\in\Cl(\bx)$. If $j\notin U$, then $(j,s_{l+1})\in\bx$ and the chain $j<s_{l+1}<\cdots<s_m=k$ witnesses the relation $(j,k)\in\Cl(\bx)$.

This completes the proof that $\Cl(\bx)$ belongs to~$\sD_U(n)$. Since~$\Cl(\bx)$ is transitive, it thus belongs to~$\sA_U(n)$.
\end{proof}

\begin{corollary}\label{C:Int(Uclosed)}
The set $\sA_U(n)$ is a $(0,1)$-sublattice of~$\sP(n)$. The meet and the join of elements $\bx,\by\in\sA_U(n)$ are given by $\bx\wedge\by=\bx\cap\by$ and $\bx\vee\by=\Cl(\bx\cup\by)$, respectively.
\end{corollary}

{}From Theorem~\ref{T:Nath00} and Corollary~\ref{C:Int(Uclosed)} it
follows immediately that~\emph{$\sA_U(n)$ is a bounded lattice, for
  each $U\subseteq[n]$}. We shall verify in
Proposition~\ref{P:CambCong} that the lattices~$\sA_U(n)$ are exactly
the \emph{Cambrian lattices of type~A} introduced in Reading~\cite{Read06}.

\section{A subdirect decomposition of the permutohedron}\label{S:SubdDec}

In this section we shall strengthen Corollary~\ref{C:Int(Uclosed)} by
proving that every lattice~$\sA_U(n)$ is a \emph{retract} (with
respect to the lattice operations) of the
permutohedron~$\sP(n)$. This result is obtained by
introducing the general definition of \emph{join-fitness} of a finite
\jzus\ within a larger finite lattice (Definition~\ref{D:Fits}) and
proving that~$\sA_U(n)$ join-fits within~$\sP(n)$. We shall also prove
(Proposition~\ref{P:DecompPU(n)}) that every permutohedron~$\sP(n)$ is
a subdirect product of the corresponding~$\sA_U(n)$ and that the
lattices~$\sA_U(n)$ are exactly the Cambrian lattices of type~A
(Proposition~\ref{P:CambCong}).

Throughout this section we shall fix a positive integer~$n$.

The following lemma gives a convenient description of the \jirr\ elements of~$\sA_U(n)$, which involves the restriction operation defined in~\eqref{Eq:RestrSubset}. Its proof is a straightforward exercise.

\begin{lemma}\label{L:JiisPUn}
For any $(i,j)\in\cI_n$, the least element of~$\sA_U(n)$ containing~$(i,j)$ as an element is $\pUji{i}{j}{U}=\pji{i}{j}{U\res[i,j]}$. Consequently,
 \[
 \J(\sA_U(n))=\setm{\pUji{i}{j}{U}}{(i,j)\in\cI_n}\,.
 \]
\end{lemma}

\begin{all}{Notational convention}
For the case $U=[n]$, we shall write $\seq{i,j}$ instead of $\pUji{i}{j}{[n]}$, the \jirr\ elements of the Tamari lattice~$\sA(n)$.
\end{all}

The lattices $\sA(4)=\sA_{[4]}(4)$ and $\sA_{\set{3}}(4)$ are represented on the left hand side and right hand side of Figure~\ref{Fig:Cambrian}, respectively. On these pictures, we mark the \jirr\ elements by doubled circles and we write~$ij$ instead of~$\pUji{i}{j}{U}$.

\begin{figure}[htb]
\includegraphics{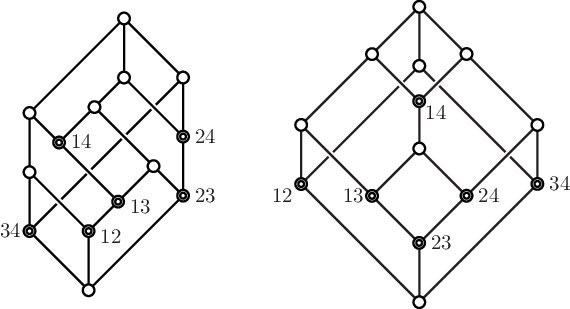}
\caption{The lattices $\sA(4)$ and $\sA_{\set{3}}(4)$}
\label{Fig:Cambrian}
\end{figure}

In order to establish that every~$\sA_U(n)$ is a retract of the
corresponding permutohedron~$\sP(n)$, it is convenient to introduce
the following concept.

\begin{definition}\label{D:Fits}
We say that a \jzu-subsemilattice~$K$ of a finite lattice~$L$ \emph{join-fits within~$L$} if
 \[
 (\forall(p,q)\in\J(K)\times\J(L))(p\bD_L q\Rightarrow q\in K)\,.
 \]
\end{definition}

\begin{lemma}\label{L:NiceLattQuot}
Let $K$ be a lattice that join-fits within a finite lattice~$L$. Then the lower projection map $(\pi\colon L\to K$, $y\mapsto\text{largest }x\in K\text{ such that }x\leq y)$ is a surjective lattice homomorphism.
\end{lemma}

\begin{proof}
  It is obvious that~$\pi$ is a surjective meet-homomorphism. Now let
  $y_0,y_1\in L$, we must prove that $\pi(y_0\vee
  y_1)\leq\pi(y_0)\vee\pi(y_1)$. It suffices to prove that for each
  $p\in\J(K)$, if $p\leq y_0\vee y_1$, then
  $p\leq\pi(y_0)\vee\pi(y_1)$. The result is trivial if either $p\leq
  y_0$ or $p\leq y_1$, so suppose from now on that $p\nleq y_0$ and
  $p\nleq y_1$. Since $p\leq y_0\vee y_1$, there exists a minimal nontrivial join-cover~$I$ of~$p$ in~$L$ refining $\set{y_0,y_1}$. It follows that $I=I_0\cup I_1$, where we set $I_k=I\dnw y_k$ for each
  $k\in\set{0,1}$. For each $k<2$ and each $q\in I_k$, the relation
  $p\bD_Lq$ holds, thus, by assumption, $q\in K$. {}From $q\leq y_k$
  and $q\in K$ it follows that $q\leq\pi(y_k)$; thus $\bigvee
  I_k\leq\pi(y_k)$ as well. Therefore, $p\leq\bigvee I_0\vee\bigvee
  I_1\leq\pi(y_0)\vee\pi(y_1)$.
\end{proof}

\begin{proposition}\label{P:PUnjoinfitsPn}
  The lattice~$\sA_U(n)$ join-fits within~$\sP(n)$, for every positive
  integer~$n$ and every $U\subseteq[n]$. Consequently, $\sA_U(n)$ is a
  lattice-theoretical retract of~$\sP(n)$.
\end{proposition}

\begin{proof}
  It follows from Corollary~\ref{C:Int(Uclosed)} that~$\sA_U(n)$ is a
  $(0,1)$-sublattice (thus, \emph{a fortiori}, a \jzu-subsemilattice)
  of~$\sP(n)$. Now let $\bp\in\J(\sA_U(n))$ and $\bq\in\J(\sP(n))$
  such that $\bp\bD_{\sP(n)}\bq$. By Lemma~\ref{L:JiisPUn}, there
  exists $(a,b)\in\cI_n$ such that $\bp=\pji{a}{b}{U\res[a,b]}$. By
  Lemma~\ref{L:DescJ(Permn)} and Proposition~\ref{P:DrelPerm}, there
  exists $(c,d)\in\cI_n$ such that $[c,d]\subsetneqq[a,b]$ and
  $\bq=\pji{c}{d}{(U\res[a,b])\res[c,d]}$. Thus
  $\bq=\pji{c}{d}{U\res[c,d]}$ belongs to~$\J(\sA_U(n))$, which
  completes the proof of the join-fitness statement. The retractness
  statement then follows from Corollary~\ref{C:Int(Uclosed)}.
\end{proof}

    For $U=[n]$, we obtain the following corollary. 
    \begin{corollary}\label{C:PUnjoinfitsPn}
The Tamari lattice~$\sA(n)$ is a lattice-theoretical retract of the permutohedron~$\sP(n)$, for every positive integer~$n$.
    \end{corollary}

    We shall verify soon (cf. Proposition~\ref{P:CambCong}) that the
    lattices~$\sA_U(n)$ are exactly the Cambrian lattices associated
    to~$\sP(n)$. Assuming this result,
    Proposition~\ref{P:PUnjoinfitsPn} gives an alternative proof of
    Reading's result \cite[Theorem~6.5]{Read06} that the Cambrian
    lattices of type~A are retracts of the corresponding
    permutohedra. Let us notice that, while Reading simply states that
    Cambrian lattices of type~A are sublattices of the corresponding
    permutohedra, his proof actually exhibits these sublattices as
    retracts. The analogous statement for Tamari lattices
    (Corollary~\ref{C:PUnjoinfitsPn}) was already observed in
    Bj\"orner and Wachs \cite[Theorem~9.6]{BjWa97}.

The method of proof of Lemma~\ref{P:PUnjoinfitsPn} yields immediately the following.

\begin{lemma}\label{L:DRelPU}
  The equality $\J(\sA_U(n))=\sA_U(n)\cap\J(\sP(n))$ holds, and
  $\bp\bD_{\sA_U(n)}\bq$ if{f} $\bp\bD_{\sP(n)}\bq$, for all
  $\bp,\bq\in\J(\sA_U(n))$. Furthermore,
  $\pUji{a}{b}{U}\bD_{\sA_U(n)}\pUji{c}{d}{U}$ if{f}
  $[c,d]\subsetneqq[a,b]$, for all $(a,b),(c,d)\in\cI_n$.
\end{lemma}

Denote by~$\pi_U\colon\sP(n)\twoheadrightarrow\sA_U(n)$ the canonical projection (defined by $\pi_U(\bx)=$ largest element of~$\sA_U(n)$ contained in~$\bx$). By Proposition~\ref{P:PUnjoinfitsPn}, $\pi_U$ is a lattice homomorphism.

\begin{proposition}\label{P:DecompPU(n)}
Every lattice~$\sA_U(n)$ is subdirectly irreducible, and the diagonal map $\pi\colon\sP(n)\to\prod\famm{\sA_U(n)}{U\subseteq[n]}$, $\bx\mapsto\famm{\pi_U(\bx)}{U\subseteq[n]}$ is a subdirect product decomposition of the permutohedron~$\sP(n)$.
\end{proposition}

\begin{proof}
  It follows from Lemma~\ref{L:DRelPU} that $\pUji{1}{n}{U}$ is the
  least \jirr\ element of~$\sA_U(n)$ with respect to the transitive
  closure of the relation~$\bD_{\sA_U(n)}$. Consequently, by Freese,
  Je\v{z}ek, and Nation \cite[Corollary~2.37]{FJN}, $\sA_U(n)$ is
  subdirectly irreducible.
  
  It remains to prove that the map~$\pi$ is one-to-one. Let
  $\ba,\bb\in\sP(n)$ such that $\ba\not\subseteq\bb$. By
  Lemma~\ref{L:DescJ(Permn)}, there exists $(i,j,U)\in\cF_n$ such that
  the element $\bp=\pji{a}{b}{U}$ is contained in~$\ba$ but not
  in~$\bb$. Now $\bp=\pUji{a}{b}{U}$ belongs to $\J(\sA_U(n))$, thus
  $\bp\in\pi_U(\ba)\setminus\pi_U(\bb)$, and thus
  $\pi_U(\ba)\not\subseteq\pi_U(\bb)$.
\end{proof}

For a \jirr\ element~$p$ in a finite lattice~$L$, we set
 \begin{align*}
 \Theta_L(p)&=\text{least congruence of }L\text{ that identifies }
 p\text{ and }p_*\,,\\
 \Psi_L(p)&=\text{largest congruence of }L\text{ that does not identify }
 p\text{ and }p_*\,.
 \end{align*}
We shall also write $\Theta(p)$, $\Psi(p)$ in case the lattice~$L$ is understood. It follows from Freese, Je\v{z}ek, and Nation \cite[Theorem~2.30]{FJN} that the \jirr\ congruences of~$L$ are exactly those of the form~$\Theta_L(p)$, while the \mirr\ congruences of~$L$ are exactly those of the form~$\Psi_L(p)$.

The following lemma gives a description of the kernel of~$\pi_U$ in terms of the \jirr\ elements of~$\sP(n)$.

\begin{lemma}\label{L:CongPU}
The kernel of~$\pi_U$ is equal to $\Psi_{\sP(n)}(\pji{1}{n}{U})$, for each $U\subseteq[n]$.
\end{lemma}

\begin{proof}
  Since the definition of~$\sA_U(n)$ depends only of
  $U\setminus\set{1,n}$, we may assume that $1\notin U$ and $n\in U$,
  that is, $(1,n,U)\in\cF_n$. Set $\theta=\Ker\pi_U$ and
  $\bp=\pji{1}{n}{U}$. Since~$\bp$ belongs to~$\sA_U(n)$,
  $\pi_U(\bp)=\bp>\bp_*\geq\pi_U(\bp_*)$, thus
  $\bp\not\equiv\bp_*\pmod{\theta}$. Conversely, we need to prove that
  every congruence~$\psi$ of~$\sP(n)$ such that
  $\bp\not\equiv\bp_*\pmod{\psi}$ is contained in~$\theta$. We may
  assume that~$\psi$ is \jirr, so $\psi=\Theta_{\sP(n)}(\bq)$, with
  $\bq=\pji{c}{d}{V}$ for some $(c,d,V)\in\cF_n$ (cf.
  Lemma~\ref{L:DescJ(Permn)}).
  
  Denoting by~$\utr$ the reflexive and transitive closure of the
  relation~$\bD_{\sP(n)}$, $\bp\not\equiv\bp_*\pmod{\psi}$ means that
  $\bp\not\utr\bq$ (cf. Freese, Je\v{z}ek, and Nation
  \cite[Lemma~2.36]{FJN}), that is, using
  Proposition~\ref{P:DrelPerm}, $V\neq U\res[c,d]$. It follows easily
  that~$\bq$ does not belong to~$\sA_U(n)$, thus
  $\pi_U(\bq)\leq\bq_*$, $\pi_U(\bq)=\pi_U(\bq_*)$, so
  $(\bq,\bq_*)\in\theta$, that is, $\psi\subseteq\theta$.
\end{proof}

As we shall verify soon, the lattices~$\sA_U(n)$ are identical to the \emph{Cambrian lattices of type A} introduced in
Reading~\cite{Read06}. This result is, actually, already contained in results from Reading \cite{Read07a,Read07b}. We shall now give an outline of how this works.

Recall first how Cambrian lattices of type~A are defined. For an integer~$n\geq2$ (if $n=1$ then everything is trivial), we set $s_i=\begin{pmatrix}i&i+1\end{pmatrix}$ for $1\leq i<n$. The Dynkin diagram of~$\fS_n$ is the undirected graph having as vertices the~$s_{i}$ and, as edges, the pairs $\set{s_{i-1},s_{i}}$ for $i = 2,\ldots ,n-1$.
Informally, an orientation of the Dynkin diagram of~$\fS_n$
consists of a choice, for each index~$i\in[n-2]$, of an orientation
between the two vertices~$s_i$ and~$s_{i+1}$: that is,
either~$s_i\rightarrow s_{i+1}$ or $s_i\leftarrow s_{i+1}$. Hence the
orientation is encoded by a subset of $\set{2,3,\dots,n-1}$, namely
 \[
 U=\setm{i+1}{1\leq i\leq n-2\text{ and }s_i\rightarrow s_{i+1}}\,.
 \]
The \emph{Cambrian congruence} associated to~$U$ is the lattice congruence~$\eta$ of~$\sP(n)$ generated by all pairs $s_{i+1}\equiv s_{i+1}s_i\pmod{\eta}$ if $i+1\in U$, and $s_i\equiv s_is_{i+1}\pmod{\eta}$ if $i+1\notin U$. Now, identifying a permutation with its set of inversions as defined in Section~\ref{S:BasicPerm}, we obtain that the Cambrian congruence~$\eta$ is generated by the pairs
 \begin{align*}
 \set{(i+1,i+2)}&\equiv\set{(i+1,i+2),(i,i+2)}\pmod{\eta}\,,
 &&\text{if }i+1\in U\,,\\
 \set{(i,i+1)}&\equiv\set{(i,i+1),(i,i+2)}\pmod{\eta}\,,&&\text{if }i+1\notin U\,. 
 \end{align*}
The associated \emph{Cambrian lattice} is defined as~$\sP(n)/{\eta}$.

According to Theorems~1.1 and~1.4 in Reading~\cite{Read07b}, the
``$c$-sortable'' elements of~$\sP(n)$, where~$c$ denotes a Coxeter
word associated to the given orientation, are exactly the bottom
elements of the $c$-Cambrian congruence, denoted there by~$\Theta_c$
and identical to our congruence~$\eta$.  On the other hand, Reading
introduces in \cite[Section~4]{Read07a} the ``$c$-aligned''
elements. By \cite[Lemma~4.8]{Read07a}, the $c$-aligned elements
of~$\sP(n)$ are exactly the elements of~$\sA_U(n)$. By Reading
\cite[Theorem~4.1]{Read07a}, ``$c$-sortable'' is the same as
``$c$-aligned''. This shows that the Cambrian lattices of type~A are
exactly the lattices~$\sA_U(n)$. We give, for the reader's
convenience, a direct proof of that fact below.

\begin{proposition}\label{P:CambCong}
The Cambrian congruence associated to~$U$ is the kernel of~$\pi_U$. Consequently, the associated Cambrian lattice is~$\sA_U(n)$.
\end{proposition}

\begin{proof}
Set again $\theta=\Ker\pi_U$.

Let $i\in[n-2]$. Suppose first that $i+1\in U$. For each $\bx\in\sA_U(n)$ with $\bx\subseteq\set{(i+1,i+2),(i,i+2)}$, the possibility that $(i,i+2)\in\bx$ is ruled out for it would imply (as $i+1\in U$) that $(i,i+1)\in\bx$, \contr; hence $\bx\subseteq\set{(i+1,i+2)}$, and hence
 \[
 \set{(i+1,i+2)}\equiv\set{(i+1,i+2),(i,i+2)}\pmod{\theta}\,.
 \]
Similarly, we can prove that if $i+1\notin U$, then
 \[
 \set{(i,i+1)}\equiv\set{(i,i+1),(i,i+2)}\pmod{\theta}\,.
 \]
It follows that~$\theta$ contains~$\eta$.

In order to establish the converse containment, remember from
Lemma~\ref{L:CongPU} that~$\theta$ is generated by all~$\Theta(\bq)$,
where $\bq=\pji{c}{d}{V}\in\J(\sP(n))$ with $(c,d,V)\in\cF_n$ and
$V\neq U\res[c,d]$. Hence it suffices to prove that
$\bq\equiv\bq_*\pmod{\eta}$ for each such~$\bq$. We separate cases. If
$U\res[c,d]\not\subseteq V$, pick~$i$ in the difference; observe that
$c<i<d$. {}From $i\in U$ it follows that $\set{(i,i+1)}\equiv\set{(i,i+1),(i-1,i+1)}\pmod{\eta}$,
that is, as $i\notin V$, $\pUji{i}{i+1}{V}\equiv\pUji{i-1}{i+1}{V}\pmod{\eta}$.
Thus, setting $\pUji{k}{k}{V}=\es$ for each~$k$, we get
 \[
 \bq=\pUji{c}{d}{V}\leq\pUji{c}{i-1}{V}\vee\pUji{i-1}{i+1}{V}\vee\pUji{i+1}{d}{V}\equiv\bx\pmod{\eta}
 \]
 where we set
 $\bx=\pUji{c}{i-1}{V}\vee\pUji{i}{i+1}{V}\vee\pUji{i+1}{d}{V}$.
 {}From $(c,d)\notin\bx$ it follows that $\bq\not\subseteq\bx$, thus
 $\bq\equiv\bq_*\pmod{\eta}$, as desired. The proof in case
 $V\not\subseteq U\res[c,d]$ is similar, now picking an index~$i\in
 V\setminus(U\res[c,d])$ and obtaining, this time,
 elements~$\by,\by'\in\sP(n)$ such that $(i,d)\notin\by$ and
 $\bq\leq\by'\equiv\by\pmod{\eta}$.
\end{proof}

Since the elements $\pUji{1}{n}{U}$ are exactly the \emph{minimal} elements of $\J(\sP(n))$ with respect to the transitive closure~$\utr$ of the join-dependency relation, a straightforward application of Freese, Je\v{z}ek, and Nation \cite[Lemma~2.36]{FJN} yields the following.

\begin{corollary}\label{C:CambCong}
The Cambrian lattices of type~A are exactly the quotients of permutohedra by their minimal \mirr\ congruences.
\end{corollary}

The following consequence of Lemma~\ref{L:CongPU} can be obtained, \emph{via} Proposition~\ref{P:CambCong}, from Reading \cite[Theorem~3.5]{Read06}. We show here an easy, direct argument.

\begin{corollary}
\label{C:PU(n)dual}
The lattices~$\sA_U(n)$ and~$\sA_{[n]\setminus U}(n)$ are dually isomorphic, for each $U\subseteq[n]$.
\end{corollary}

\begin{proof}
Denote by $\gamma\colon\sP(n)\to\sP(n)$, $\bx\mapsto\bx^\cpl$ the canonical dual automorphism (cf. Proposition~\ref{P:Perm(n)compl}).
Again, we may assume that $1\notin U$ and $n\in U$.
Set $\bp=\pji{1}{n}{U}$, $\widetilde{U}=(\oo{1,n}\setminus U)\cup\set{n}$ and $\bq=\pji{1}{n}{\widetilde{U}}$. As observed after the statement of Lemma~\ref{L:ArrPrecJirr}, $\kappa_{\sP(n)}(\bp)=\gamma(\bq)$. It follows that the prime interval $[\bp_*,\bp]$ projects up to the interval $[\gamma(\bq),\gamma(\bq)^*]$, hence, as~$\gamma$ is a dual automorphism and using Lemma~\ref{L:CongPU}, $\Ker\pi_U=\gamma(\Ker\pi_{\widetilde{U}})$, and hence $\sA_U(n)\cong\sP(n)/{\Ker\pi_U}$ is dually isomorphic to $\sP(n)/{\Ker\pi_{\widetilde{U}}}\cong\sA_{\widetilde{U}}(n)=\sA_{[n]\setminus U}(n)$.
\end{proof}

In particular, since the Tamari lattice~$\sA(n)$ is self-dual, it is isomorphic to both~$\sA_\es(n)$ and to~$\sA_{[n]}(n)$.

\section{The Gazpacho identities}\label{S:Gazpacho}

In this section we shall construct an infinite collection of lattice-theoretical identities, the \emph{Gazpacho identities}, and prove that these identities hold in every Tamari lattice (Theorem~\ref{T:GzpinAssoc}).

We denote by~$\SS$ the set of all finite sequences $\vec{m}=(m_1,\dots,m_d)$ of positive integers with $d\geq2$, and we set
 \[
 \fF(\vec m)=\prod\famm{[m_i]}{1\leq i\leq d}\,,
 \quad\text{for each }\vec m\in\SS\,.
 \]
We also define terms~$\sa_i$, $\tb_i$, $\se_{\vec{m}}$, $\se^*_{\vec{m}}$ in the variables~$\sa_{i,j}$ and~$\sb_i$ (for $1\leq i\leq d$ and $1\leq j\leq m_i$) by
 \begin{align}
 \sa_i&=\bigvee_{j=1}^{m_i}\sa_{i,j}\,,&
 \tb_i&=\Bigl(\bigvee_{i'=1}^d\sb_{i'}\Bigr)\wedge(\sa_i\vee\sb_i)
 \qquad(\text{for }1\leq i\leq d)\,,\label{Eq:aitbi}\\ 
 \se_{\vec{m}}&=\bigwedge_{i=1}^d(\sa_i\vee\sb_i)\,,&
 \se^*_{\vec{m}}&=\Bigl(\bigvee_{i'=1}^d\sb_{i'}\Bigr)
 \wedge\se_{\vec{m}}=\bigwedge_{i=1}^d\tb_i\,.\notag
 \end{align}
 Further, we define lattice terms~$\sf^{\sigma,\tau}_i$, for $2\leq
 i\leq d$ and $(\sigma,\tau)\in\fS_d\times\fF(\vec{m})$, by downward
 induction on~$i$ (for $2\leq i<d$), by
 \begin{align}
 \sf^{\sigma,\tau}_d&=(\sa_{\sigma(d),\tau\sigma(d)}\vee
 \tb_{\sigma(1)})\wedge(\sa_{\sigma(d)}\vee\sb_{\sigma(d)})\,,\label{Eq:fstd}\\
 \sf^{\sigma,\tau}_i&=(\sa_{\sigma(i),\tau\sigma(i)}\vee\tb_{\sigma(1)})\wedge
 (\sa_{\sigma(i)}\vee\sb_{\sigma(i)})\wedge\bigwedge_{i<j\leq d}
 \bigl(\sa_{\sigma(i),\tau\sigma(i)}\vee\sf^{\sigma,\tau}_j\bigr)\,.\notag
 \end{align}
Let \Gzp{\vec{m}} (the \emph{Gazpacho identity with index~$\vec{m}$}) be the following lattice-theoretical identity, in the variables~$\sa_{i,j}$ and~$\sb_i$, for $1\leq i\leq d$ and $1\leq j\leq m_i$:
 \begin{equation}
 \se_{\vec{m}}\leq\se^*_{\vec{m}}\vee\bigvee\Famm{\sf^{\sigma,\tau}_2}
 {(\sigma,\tau)\in\fS_d\times\fF(\vec{m})}\,.\tag{\Gzp{\vec{m}}}
 \end{equation}

\begin{theorem}\label{T:GzpinAssoc}
Every Tamari lattice satisfies \Gzp{\vec{m}} for each $\vec{m}\in\SS$.
\end{theorem}

\begin{proof}
Let $\ell$ be a positive integer.
Set $\vec{m}=(m_1,\dots,m_d)$ with $d\geq2$ and let~$\ba_{i,j}$ and~$\bb_i$ (for $1\leq i\leq d$ and $1\leq j\leq m_i$) be elements of $\sA(\ell)$. We define $\bb=\bigvee_{i=1}^d\bb_i$, and, applying the lattice polynomials defined above, elements $\ba_i$, $\tilde{\bb}_i$, $\be=\se_{\vec{m}}(\vec{\ba},\vec{\bb})$, $\bf^{\sigma,\tau}_i=\sf^{\sigma,\tau}_i(\vec{\ba},\vec{\bb})$, and
 \[
 \bf=(\bb\wedge\be)\vee\bigvee\Famm{\bf^{\sigma,\tau}_2}
 {(\sigma,\tau)\in\fS_d\times\fF(\vec{m})}\,.
 \]
We must prove that~$\be$ is contained in~$\bf$. Suppose otherwise and let $(x,y)\in\be\setminus\bf$ with the interval $[x,y]$ minimal with that property. For each $i\in[d]$, there exists a subdivision
 \begin{equation}\label{Eq:Subdivzij}
 x=z^i_0<z^i_1<\cdots<z^i_{n_i}=y\text{ with }(z^i_j,z^i_{j+1})\in
 \bigcup_{k=1}^{m_i}\ba_{i,k}\cup\bb_i
 \text{ for each }j<n_i\,.
 \end{equation}
We set $Z_i=\setm{z^i_j}{0\leq j\leq n_i}$, for each $i\in[d]$.
It follows from the minimality assumption on $[x,y]$ that $(x,z^i_j)\in\bf$ for each $i\in[d]$ and each $j<n_i$. Since $(x,y)\notin\bf$, it follows that $(z^i_j,y)\notin\bf$; in particular, $(z^i_{n_i-1},y)\notin\bf$. However, from $\bf^{\sigma,\tau}_2\geq\ba_{\sigma(2),\tau\sigma(2)}$ for each $(\sigma,\tau)\in\fS_d\times\fF(\vec{m})$ it follows that $\ba_i\leq\bf$, thus, \emph{a fortiori}, $\bigcup_{k=1}^{m_i}\ba_{i,k}\subseteq\bf$, and thus, by~\eqref{Eq:Subdivzij}, $(z^i_{n_i-1},y)\in\bb_i$.

Let $i\in[d]$. Since $\bb_i\subseteq\tilde{\bb}_i$, there exists a least $z_i\in Z_i\setminus\set{y}$ such that $(z_i,y)\in\tilde{\bb}_i$. If $z_i=x$, then $(x,y)$ belongs to $\tilde{\bb}_i\wedge\be=\bb\wedge\be$, thus to~$\bf$, \contr; hence $x<z_i$. Pick $i_1\in[d]$ such that $z_{i_1}\leq z_i$ for each $i\in[d]$. Denote by~$s_i$ the largest element of~$Z_i\ddnw z_{i_1}$ and by~$s'_i$ the successor of~$s_i$ in~$Z_i$, for each $i\in[d]\setminus\set{i_1}$. There exists a permutation $\sigma\in\fS_d$ such that $\sigma(1)=i_1$ and $s_{\sigma(2)}\leq s_{\sigma(3)}\leq\cdots\leq s_{\sigma(d)}$.

Suppose that $(s_i,s'_i)\in\bb_i$, for some $i\in[d]\setminus\set{\sigma(1)}$. {}From $s_i<z_{\sigma(1)}\leq s'_i$ it follows that $(s_i,z_{\sigma(1)})\in\bb_i$, thus, as $(z_{\sigma(1)},y)\in\tilde{\bb}_{\sigma(1)}$, we obtain that $(s_i,y)\in\bb_i\vee\tilde\bb_{\sigma(1)}$, thus $(s_i,y)\in\bb$. {}From $\set{s_i,y}\subseteq Z_i$ it follows that $(s_i,y)\in\ba_i\vee\bb_i$, and so $(s_i,y)\in\bb\wedge(\ba_i\vee\bb_i)=\tilde{\bb}_i$, \contr\ as $s_i<z_i$. Therefore, $(s_i,s'_i)\notin\bb_i$, and therefore, by~\eqref{Eq:Subdivzij}, there exists $\tau(i)\in[m_i]$ such that $(s_i,s'_i)\in\ba_{i,\tau(i)}$. Since $s_i<z_{\sigma(1)}\leq s'_i$, we also get $(s_i,z_{\sigma(1)})\in\ba_{i,\tau(i)}$.

{}From $(s_{\sigma(i)},z_{\sigma(1)})\in\ba_{\sigma(i),\tau\sigma(i)}$ and $(z_{\sigma(1)},y)\in\tilde{\bb}_{\sigma(1)}$ it follows that $(s_{\sigma(i)},y)\in\ba_{\sigma(i),\tau\sigma(i)}\vee\tilde{\bb}_{\sigma(1)}$. Moreover, from $\set{s_{\sigma(i)},y}\subseteq Z_{\sigma(i)}$ it follows that $(s_{\sigma(i)},y)\in\ba_{\sigma(i)}\vee\bb_{\sigma(i)}$, and therefore
 \begin{equation}\label{Eq:Intermssigiyin}
 (s_{\sigma(i)},y)\in(\ba_{\sigma(i),\tau\sigma(i)}\vee\tilde{\bb}_{\sigma(1)})
 \wedge(\ba_{\sigma(i)}\vee\bb_{\sigma(i)})\,.
 \end{equation}
Now we prove, by downward induction on~$i$, that $(s_{\sigma(i)},y)\in\bf^{\sigma,\tau}_i$, for each $i\in[2,d]$. The case $i=d$ follows readily from~\eqref{Eq:Intermssigiyin}. Now suppose that $2\leq i<d$ and that $(s_{\sigma(j)},y)\in\bf^{\sigma,\tau}_j$ for each~$j$ with $i<j\leq d$. Fix such a~$j$. {}From $(s_{\sigma(i)},z_{\sigma(1)})\in\ba_{\sigma(i),\tau\sigma(i)}$ and $s_{\sigma(i)}\leq s_{\sigma(j)}<z_{\sigma(1)}$ it follows that $(s_{\sigma(i)},s_{\sigma(j)})\in\Delta_{\ell}\cup\ba_{\sigma(i),\tau\sigma(i)}$. By induction hypothesis, it follows that $(s_{\sigma(i)},y)\in\ba_{\sigma(i),\tau\sigma(i)}\vee\bf^{\sigma,\tau}_j$. Therefore, meeting the right hand side of this relation over all~$j$ and then with the right hand side of~\eqref{Eq:Intermssigiyin}, we obtain that $(s_{\sigma(i)},y)\in\bf^{\sigma,\tau}_i$, as desired.

In particular, $(s_{\sigma(2)},y)\in\bf^{\sigma,\tau}_2\subseteq\bf$. By the minimality assumption on the interval $[x,y]$, the pair $(x,s_{\sigma(2)})$ belongs to~$\bf$, and so $(x,y)\in\bf$, \contr.
\end{proof}

Due to the complexity of the identities \Gzp{\vec{m}} for general~$\vec{m}$, we shall study some of their much simpler consequences instead.

\section{A first nontrivial identity for all Tamari lattices}\label{S:Veg1}

In this section we shall prove that the simplest Gazpacho identity
does not hold in the Cambrian lattice~$\sA_{\set{3}}(4)$, thus
providing our first counterexample to Geyer's conjecture.

Consider the identity \Gzp{\vec{m}} with $\vec{m}=(1,1)$. It has the
four variables $\sa_1$, $\sa_2$, $\sb_1$, $\sb_2$, it involves the
terms $\tb_i=(\sb_1\vee\sb_2)\wedge(\sa_i\vee\sb_i)$, for
$i\in\set{1,2}$, and $\se=(\sa_1\vee\sb_1)\wedge(\sa_2\vee\sb_2)$. Since
$\fF(\vec{m})$ is a singleton, the superscript~$\tau$ becomes
irrelevant in the term~$\sf^{\sigma,\tau}_d$ given in~\eqref{Eq:fstd}
(for $d=2$), so we omit it, and then
 \[
 \sf^{\sigma}_2=(\sa_{\sigma(2)}\vee\tb_{\sigma(1)})\wedge
 (\sa_{\sigma(2)}\vee\sb_{\sigma(2)})\,,\quad\text{for each }\sigma\in\fS_2\,.
 \]
Consequently, \Gzp{1,1} is equivalent to the following identity:
 \[
 (\sa_1\vee\sb_1)\wedge(\sa_2\vee\sb_2)\leq
 (\tb_1\wedge\tb_2)\vee\bigl((\sa_1\vee\tb_2)\wedge(\sa_1\vee\sb_1)\bigr)
 \vee\bigl((\sa_2\vee\tb_1)\wedge(\sa_2\vee\sb_2)\bigr)\,.
 \]
Now observing that $\sa_i\vee\sb_i=\sa_i\vee\tb_i$ in every lattice, we can cancel out the term $\tb_1\wedge\tb_2$ and thus we obtain the following equivalent form of \Gzp{1,1}, which we shall denote by~\Veg{1}:
 \begin{equation}
 (\sa_1\vee\sb_1)\wedge(\sa_2\vee\sb_2)\leq
 \bigl((\sa_1\vee\sb_1)\wedge(\sa_1\vee\tb_2)\bigr)
 \vee\bigl((\sa_2\vee\tb_1)\wedge(\sa_2\vee\sb_2)\bigr)\,.\tag*{\Veg{1}}
 \end{equation}
Hence, as a consequence of Theorem~\ref{T:GzpinAssoc}, we obtain the following result.

\begin{corollary}\label{C:Asso2VEG1}
Every Tamari lattice satisfies \Veg{1}.
\end{corollary}
\begin{theorem}\label{T:Perm4NotEmb}
The permutohedron~$\sP(4)$ does not satisfy the identity \Veg{1}. In particular, it has no lattice embedding into any Tamari lattice.
\end{theorem}

\begin{proof}
  By using Proposition~\ref{P:DecompPU(n)}, it suffices to prove
  that~$\sA_U(4)$ does not satisfy \Veg{1} for a suitable
  $U\subseteq[4]$. Take $U=\set{3}$ and define elements of
  $\sA_U(4)$ by $\ba_1=\seq{1,3}_U$, $\ba_2=\seq{2,4}_U$,
  $\bb_1=\seq{3,4}_U$,
  and $\bb_2=\seq{1,2}_U$. Hence $\ba_1=\set{(1,3),(2,3)}$,
  $\ba_2=\set{(2,3),(2,4)}$, $\bb_1=\set{(3,4)}$, and
  $\bb_2=\set{(1,2)}$. Furthermore, it is straightforward to verify
  that
 \begin{align*}
 \ba_1\vee\bb_1&=\set{(1,3),(1,4),(2,3),(2,4),(3,4)}\,,\\
 \ba_2\vee\bb_2&=\set{(1,2),(1,3),(1,4),(2,3),(2,4)}\,,\\
 \ba_1\vee\bb_2&=\set{(1,2),(1,3),(2,3)}\,,\\
 \ba_2\vee\bb_1&=\set{(2,3),(2,4),(3,4)}\,,\\
 \ba_1\vee\ba_2&=\set{(1,3),(2,3),(2,4)}\,,
 \end{align*}
 thus
 \begin{align*}
   \tilde{\bb}_j&=\bb_j\,,&&\text{for all }j\in\set{1,2}\,,\\
   (\ba_i\vee\tilde{\bb}_1)\wedge(\ba_i\vee\tilde{\bb}_2)&=\ba_i\,,
   &&\text{for all }i\in\set{1,2}\,,\\
   (\ba_1\vee
   \bb_1)\wedge(\ba_2\vee\bb_2)&=\set{(1,3),(1,4),(2,3),(2,4)}\,.
 \end{align*}
 In particular, for that particular instance, \Veg{1} is not
 satisfied.
\end{proof}

\begin{remark}\label{Rk:NonEmbAlt}
  The proof of Theorem~\ref{T:Perm4NotEmb} shows that the Cambrian
  lattice~$\sA_{\set{3}}(4)$ does not satisfy the identity~\Veg{1}.
  Hence, by Corollary~\ref{C:PU(n)dual}, the Cambrian
  lattice~$\sA_{\set{3}}(4)=\sA_{[4]\setminus\set{2}}(4)$ does not
  satisfy the dual of the identity~\Veg{1}. In particular,
  $\sA_{\set{2}}(4)$ cannot be embedded into any Tamari lattice,
  either. The lattice~$\sA_{\set{3}}(4)$ is represented on the right
  hand side of Figure~\ref{Fig:Cambrian}.
\end{remark}

\begin{remark}\label{Rk:P(4)optimal}
Observe that for positive integers~$m$ and~$n$, there is a lattice embedding from the product $\sA(m)\times\sA(n)$ into~$\sA(m+n)$, obtained by sending $(\bx,\by)$ to $\bx\cup\by'$ where $\by'=\setm{(m+i,m+j)}{(i,j)\in\by}$. (A similar comment applies to embedding $\sP(m)\times\sP(n)$ into~$\sP(m+n)$.) Since the permutohedron~$\sP(3)$ is a subdirect product of two copies of the five-element modular nondistributive lattice~$\sN_5$ (see Figure~\ref{Fig:P3N5}) and~$\sN_5\cong\sA(3)$, it follows that~$\sP(3)$ embeds into $\sA(3)\times\sA(3)$, thus into~$\sA(6)$.

\begin{figure}[htb]
\includegraphics{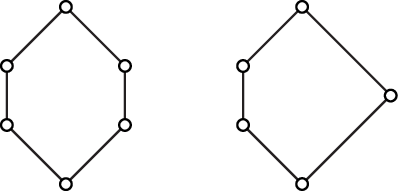}
\caption{The lattices $\sP(3)$ and~$\sN_5$}
\label{Fig:P3N5}
\end{figure}
\end{remark}

\section{Another identity for all Tamari lattices}\label{S:Veg2}

In this section we shall prove that a weakening of a certain Gazpacho identity fails in the lattice~$\sB(2,2)$ (Corollary~\ref{C:B22Config}), thus providing our second counterexample to Geyer's conjecture.

Consider the Gazpacho identity \Gzp{\vec{m}}, where $\vec{m}=(2,2)$, in which we substitute to both variables~$\sa_{1,j}$ and~$\sa_{2,j}$ the variable~$\sa_j$ (not to be confused with the lattice term~$\sa_i$ introduced in~\eqref{Eq:aitbi}), for $j\in\set{1,2}$. By arguing in a similar manner as at the beginning of Section~\ref{S:Veg1}, we see that the resulting identity is equivalent to the following identity, which we shall denote by \Veg{2}:
 \begin{equation}
 (\sa_1\vee\sa_2\vee\sb_1)\wedge(\sa_1\vee\sa_2\vee\sb_2)=
 \bigvee_{i,j\in\set{1,2}}
 \bigl((\sa_i\vee\tb_j)\wedge(\sa_1\vee\sa_2\vee\sb_{3-j})\bigr)\,,
 \tag*{\Veg{2}}
 \end{equation}
with the lattice terms $\tb_j=(\sb_1\vee\sb_2)\wedge(\sa_1\vee\sa_2\vee\sb_j)$, for $j\in\set{1,2}$. Hence, as a consequence of Theorem~\ref{T:GzpinAssoc}, we obtain the following.

\begin{theorem}\label{T:Asso2VEG2}
Every Tamari lattice satisfies \Veg{2}.
\end{theorem}
 
For natural numbers~$m$ and~$n$, we denote by $\sB(m,n)$ the lattice
obtained by doubling the join of~$m$ atoms in the
$(m+n)$-atom Boolean lattice. It can be obtained by adding a new
element~$\bp$ to the Boolean lattice on $m+n$ atoms~$\ba_1$, \dots,
$\ba_m$, $\bb_1$, \dots, $\bb_n$, with the extra relations $\ba_i<\bp$
(for $1\leq i\leq m$) and $\bp<\ba_1\vee\cdots\vee \ba_m\vee \bb_j$
(for $1\leq j\leq n$). The lattices~$\sB(1,3)$ and~$\sB(2,2)$ are
represented in Figure~\ref{Fig:B13B22}, with their \jirr\ elements
marked by doubled circles.

\begin{figure}[htb]
\includegraphics{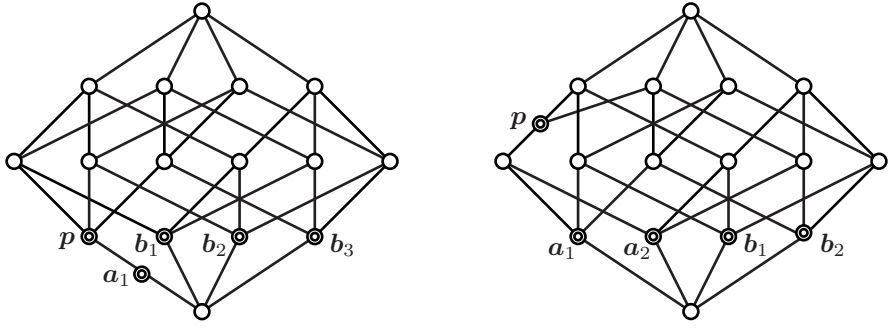}
\caption{The lattices $\sB(1,3)$ and $\sB(2,2)$}
\label{Fig:B13B22}
\end{figure}

The lattice~$\sB(m,n)$ is a so-called \emph{almost distributive lattice} (cf. Jipsen and Rose \cite[Lemma~4.11]{JiRo}), and it is subdirectly irreducible (cf. \cite[Theorem~4.17]{JiRo}). It is obtained by doubling a point from a finite Boolean lattice, thus it is bounded (cf. Freese, Je\v{z}ek, and Nation \cite[Theorem~2.44]{FJN}).

The class of lattices of the form~$\sB(m,n)$ is self-dual:

\begin{lemma}\label{L:DualBmn}
The lattices $\sB(m,n)$ and $\sB(n,m)$ are dually isomorphic, for all natural numbers~$m$ and~$n$.
\end{lemma}

\begin{proof}
Let~$x$ and~$y$ be disjoint sets of cardinality~$m$ and~$n$, respectively, and denote by~$\sB(x,y)$ the lattice obtained by doubling~$x$ in the powerset lattice~$\Pow(x\cup y)$ of~$x\cup y$. Hence $\sB(x,y)=\Pow(x\cup y)\cup\set{\bp}$ and $\sB(y,x)=\Pow(x\cup y)\cup\set{\bq}$, for new elements~$\bp$ and~$\bq$ such that
 \begin{gather*}
 x<\bp\text{ and }\bp<x\cup\set{j}\text{ for each }j\in y\,,\\
 y<\bq\text{ and }\bq<\set{i}\cup y\text{ for each }i\in x\,.
 \end{gather*}
Define a map $\varphi\colon\sB(x,y)\to\sB(y,x)$ by $\varphi(x)=\bq$, $\varphi(\bp)=y$, and $\varphi(z)=(x\cup y)\setminus z$ for each $z\in\Pow(x\cup y)\setminus\set{x}$. Then~$\varphi$ is a dual isomorphism. Now $\sB(m,n)\cong\sB(x,y)$ and $\sB(n,m)\cong\sB(y,x)$.
\end{proof}

The evaluations in~$\sB(2,2)$ of the lattice terms~$\tb_1$ and~$\tb_2$ at the quadruple $(\ba_1,\ba_2,\bb_1,\bb_2)$ are~$\bb_1$ and~$\bb_2$, respectively, so the left hand side of~\Veg{2} is evaluated by~$\bp$ while its right hand side is evaluated by~$\ba_1\vee \ba_2$. Since these two elements are distinct, we obtain the following result.

\begin{corollary}\label{C:B22Config}
The lattice~$\sB(2,2)$ does not satisfy the identity~\Veg{2}. In particular, it cannot be embedded into any Tamari lattice.
\end{corollary}

\begin{remark}\label{Rk:BmnsatVeg2}
It is not hard (although a bit tedious) to verify that $\sB(m,n)$ satisfies \Veg{1} for all non simultaneously zero natural numbers~$m$ and~$n$. In particular, $\sB(2,2)$ satisfies~\Veg{1} but not~\Veg{2} (cf. Corollary~\ref{C:B22Config}). On the other hand, $\sP(4)$ does not satisfy~\Veg{1} (cf. Theorem~\ref{T:Perm4NotEmb}) and it can be verified that it satisfies~\Veg{2}. In particular, \emph{none of the identities~\Veg{1} and~\Veg{2} implies the other}.
\end{remark}

\section{Polarized measures and \mh s to Cambrian lattices}\label{S:PolMeas}

In the present section we shall introduce a convenient tool for
handling lattice embeddings into Cambrian lattices of type~A, inspired
by the theory of Galois connections (cf. Gierz \emph{et al.}~\cite{Comp} and the duality for finite lattices sketched in Santocanale~\cite{Sant09}). We shall apply this tool by proving that $\min\set{m,n}\leq1$ implies that~$\sB(m,n)$ embeds into some Tamari lattice (Theorem~\ref{T:EmbBn01}) and that $\min\set{m,n}\leq2$ implies that~$\sB(m,n)$ embeds into some Cambrian lattice of type~A, thus in some permutohedron (Proposition~\ref{P:EmbBm2}).

We set $\so{P}=\setm{(x,y)\in P\times P}{x<y}$, for any poset~$P$. Observe that $\so{[n]}=\cI_n$.

\begin{definition}\label{D:PolMeas}
Let $L$ be a \js, let~$P$ be a poset, and let $U\subseteq P$. An \emph{$L$-valued $U$-polarized measure on~$P$} is a map $\mu\colon\so{P}\to L$ such that
\begin{enumerate}
\item $\mu(x,z)\leq\mu(x,y)\vee\mu(y,z)$;

\item $y\in U$ implies that $\mu(x,y)\leq\mu(x,z)$;

\item $y\notin U$ implies that $\mu(y,z)\leq\mu(x,z)$,
\end{enumerate}
for all $x<y<z$ in~$P$. Furthermore, we say that~$\mu$ satisfies the
\emph{V-condition} if for all $(x,y)\in\so{P}$ and all $\ba,\bb\in L$,
  \begin{align}
    \label{Condition:V}
    \tag{V}
    \text{\em if
    $\mu(x,y)\leq\ba\vee\bb$, then} & \\\notag
    \text{\em there are $m \geq 1$ and } & \text{\em 
    a subdivision $x=z_0<z_1<\cdots<z_m=y$ in~$P$ such that} \\
  \notag
  & \text{\em either~$\mu(z_i,z_{i+1})\leq\ba$ or $\mu(z_i,z_{i+1})\leq\bb$ for each $i<m$\,.}
\end{align}
If \eqref{Condition:V} holds, then we shall say that \emph{the refinement problem $\mu(x,y)\leq\ba\vee\bb$ can be solved in~$P$}.
In case $U=P$, we shall say \emph{polarized measure} instead of
$U$-polarized measure. Furthermore, if~$L$ has a least element~$0$,
then we shall often extend the $U$-polarized measures by setting
$\mu(x,x)=0$ for each $x\in P$.
\end{definition}

In all the cases that we will consider in this paper, $P$ will be a finite chain, most of the time (but not always) of the form~$[n]$ for a positive integer~$n$. For the rest of this section we shall fix a positive integer~$n$.

\begin{example}\label{Ex:PolarMeas1}
Set $L=\sA(n)$. Then the assignment $\mu\colon(x,y)\mapsto\seq{x,y}$ defines an $L$-valued polarized measure on~$[n]$. Furthermore, $\mu$ satisfies the V-condition and its range \jz-generates the lattice~$L$.
\end{example}

\begin{example}\label{Ex:PolarMeas2}
Set $L=\sP(n)$. Then the assignment $\mu\colon(x,y)\mapsto\pUji{x}{y}{U}$ defines an $L$-valued $U$-polarized measure on~$[n]$. Furthermore, $\mu$ satisfies the V-condition. However, its range does not \jz-generate~$L$ for $n\geq3$.
\end{example}

\begin{definition}\label{D:UpolDualHom}
Let $U\subseteq[n]$ and let~$L$ be a finite lattice. We say that maps $\mu\colon\cI_n\to L$ and $\varphi\colon L\to\sA_U(n)$ are \emph{dual} if $(x,y)\in\varphi(\ba)$ if{f} $\mu(x,y)\leq\ba$, for all $(x,y)\in\cI_n$ and all $\ba\in L$.
\end{definition}

We leave to the reader the straightforward proof of the following result.

\begin{proposition}\label{P:UpolDualHom}
The following statements hold, for any $U\subseteq[n]$ and any finite lattice~$L$.
\begin{enumerate}
\item If $\mu\colon\cI_n\to L$ and $\varphi\colon L\to\sA_U(n)$ are dual, then~$\mu$ is a $U$-polarized measure and~$\varphi$ is a \muh. Furthermore,
 \begin{align}
 \mu(x,y)&=\text{least }\ba\in L\text{ such that }(x,y)\in\varphi(\ba)\,,
 &&\text{for each }(x,y)\in\cI_n\,;\label{Eq:mufromphi}\\
 \varphi(\ba)&=\setm{(x,y)\in\cI_n}{\mu(x,y)\leq\ba}\,,&&
 \text{for each }\ba\in L\,.\label{Eq:phifrommu}
 \end{align}
 
 \item Every \muh\ $\varphi\colon L\to\sA_U(n)$ has a unique dual $U$-polarized measure $\mu\colon\cI_n\to L$, which is defined by the formula~\eqref{Eq:mufromphi}.

\item Every $U$-polarized measure $\mu\colon\cI_n\to L$ has a unique dual \muh\ $\varphi\colon L\to\sA_U(n)$, which is defined by the formula~\eqref{Eq:phifrommu}.
\end{enumerate}
\end{proposition}

\begin{proposition}\label{P:betwmuphi}
Let $U\subseteq[n]$, let~$L$ be a finite lattice, and let $\mu\colon\cI_n\to L$ and $\varphi\colon L\to\sA_U(n)$ be dual. The following statements hold:
\begin{enumerate}
\item $\varphi(0)=\es$ if{f} $0$ does not belong to the range of~$\mu$.

\item The range of~$\mu$ generates~$L$ as a \jz-subsemilattice if{f}~$\varphi$ is one-to-one.

\item $\mu$ satisfies the V-condition if{f} $\varphi$ is a lattice homomorphism.
\end{enumerate}
\end{proposition}

\begin{proof}
(i) is straightforward.

(ii). Suppose that~$\varphi$ is one-to-one and let $\ba\in L$. It follows from Lemma~\ref{L:JiisPUn} that there exists a decomposition $\varphi(\ba)=\bigvee_{i=1}^m\pUji{x_i}{y_i}{U}$ with a natural number~$m$ and elements $(x_i,y_i)\in\cI_n$ for $1\leq i\leq m$. Set $\ba'=\bigvee_{i=1}^m\mu(x_i,y_i)$. {}From $(x_i,y_i)\in\varphi(\ba)$ it follows that $\mu(x_i,y_i)\leq\ba$ for each~$i$; thus $\ba'\leq\ba$. Conversely, for each~$i\in[m]$, $\mu(x_i,y_i)\leq\ba'$, thus $(x_i,y_i)\in\varphi(\ba')$, and thus, by Lemma~\ref{L:JiisPUn}, $\pUji{x_i}{y_i}{U}\subseteq\varphi(\ba')$. Therefore, $\varphi(\ba)\leq\varphi(\ba')$, thus, by assumption, $\ba\leq\ba'$, and thus $\ba=\ba'$ is a join of elements of the range of~$\mu$.

Conversely, suppose that the range of~$\mu$ generates~$L$ as a \js\ and let $\ba,\bb\in L$ such that $\ba\nleq\bb$. By assumption, there exists $(x,y)\in\cI_n$ such that $\mu(x,y)\leq\ba$ and $\mu(x,y)\nleq\bb$, that is, $(x,y)\in\varphi(\ba)\setminus\varphi(\bb)$. Therefore, $\varphi$ is one-to-one.

(iii). Suppose that~$\varphi$ is a \jh\ and let $(x,y)\in\cI_n$ and $\ba,\bb\in L$ such that $\mu(x,y)\leq\ba\vee\bb$. This means that $(x,y)$ belongs to $\varphi(\ba\vee\bb)=\varphi(\ba)\vee\varphi(\bb)=\Cl(\varphi(\ba)\cup\varphi(\bb))$, thus there exists a subdivision $x=z_0<z_1<\cdots<z_m=y$ in~$[n]$ such that $(z_i,z_{i+1})\in\varphi(\ba)\cup\varphi(\bb)$ for each $i<m$; that is, either $\mu(z_i,z_{i+1})\leq\ba$ or $\mu(z_i,z_{i+1})\leq\bb$. Therefore, $\mu$ satisfies the V-condition.

Conversely, suppose that~$\mu$ satisfies the V-condition, let $\ba,\bb\in L$, and let $(x,y)\in\varphi(\ba\vee\bb)$, we must prove that $(x,y)\in\varphi(\ba)\vee\varphi(\bb)$. Since~$\mu$ and~$\varphi$ are dual, $\mu(x,y)\leq\ba\vee\bb$, thus, as~$\mu$ satisfies the V-condition, there exists a subdivision $x=z_0<z_1<\cdots<z_m=y$ in~$[n]$ such that~$\mu(z_i,z_{i+1})$ is contained in either~$\ba$ or~$\bb$ for each $i<m$; so $(z_i,z_{i+1})\in\varphi(\ba)\cup\varphi(\bb)$ for each $i<m$, and so $(x,y)\in\varphi(\ba)\vee\varphi(\bb)$.
\end{proof}

We apply Propositions~\ref{P:UpolDualHom} and~\ref{P:betwmuphi} to the following two embedding results.

\begin{theorem}\label{T:EmbBn01}
Let~$m$ and~$n$ be natural numbers. Then the lattice~$\sB(m,n)$ embeds into some Tamari lattice if{f} either $m\leq1$ or~$n\leq1$.
\end{theorem}

\begin{proof}
If $m\geq2$ and $n\geq2$, then~$\sB(2,2)$ embeds into~$\sB(m,n)$, thus, by Theorem~\ref{C:B22Config}, $\sB(m,n)$ cannot be embedded into any Tamari lattice. Hence, since every Tamari lattice is self-dual and by Lemma~\ref{L:DualBmn}, it suffices to prove that both~$\sB(m,0)$ and~$\sB(m,1)$ embed into~$\sA(m+2)$, for every positive integer~$m$. Since~$\sB(m,0)$ is distributive with~$m+1$ \jirr\ elements, the result for that lattice follows from Markowsky~\cite[page~288]{Mark92}. It remains to deal with~$\sB(m,1)$. It is convenient to describe the embedding by a polarized measure $\mu\colon\cI_{m+2}\to\sB(m,1)$. We set
 \[
 \ba_X=\bigvee_{i\in X}\ba_i\,,\quad\text{for each }X\subseteq[m]\,.
 \]
The measure $\mu\colon\cI_{m+2}\to\sB(m,1)$ is given (setting $\bb=\bb_1$) by
 \begin{align*}
 \mu(k,l)&=\ba_{[k,l-1]}\,,&&\text{for }1\leq k<l\leq m+1\,,\\
 \mu(k,m+2)&=\ba_{[k,m]}\vee\bb\,,&&\text{for }2\leq k\leq m+1\,,\\
 \mu(1,m+2)&=\bp\,. 
 \end{align*}
It is straightforward to verify that $\mu$ is a polarized V-measure. The conclusion follows then from Propositions~\ref{P:UpolDualHom} and~\ref{P:betwmuphi}.
\end{proof}

\begin{proposition}\label{P:EmbBm2}
The lattice $\sB(m,2)$ has a $(0,1)$-lattice embedding into the Cambrian lattice $\sA_{[m+2,2m+1]}(2m+2)$, for every positive integer~$m$.
\end{proposition}

\begin{proof}
We shall define the embedding \emph{via} a $[m+2,2m+1]$-polarized
measure on $[2m+2]$, by using Propositions~\ref{P:UpolDualHom}
and~\ref{P:betwmuphi}. It will be more convenient to construct the
measure on the totally ordered set $\Lambda=[-m-1,m+1]\setminus\set{0}$ 
(which is isomorphic to the interval $[2m+2]$) and to prove that it is $U$-polarized with $U=[1,m]$.

We denote by $\ba_1$, \dots, $\ba_m$, $\bb_1$, $\bb_2$, and $\bp$ the \jirr\ elements of~$\sB(m,2)$, with $\bigvee_{1\leq i\leq m}\ba_i<\bp$. We denote by $\mu\colon\so{[0,m+1]}\to\sB(m,1)$ the polarized measure given by the isomorphism $[0,m+1]\cong[1,m+2]$ and the proof of Theorem~\ref{T:EmbBn01}. In particular, $\mu(i-1,i)=\ba_i$ for $1\leq i\leq m$, $\mu(m,m+1)=\bb_1$, and $\mu(0,m+1)=\bp$. Moreover, set $A=\setm{\ba_i}{1\leq i\leq m}$.

We denote by $\sB'(m,1)$ the copy of~$\sB(m,1)$, within~$\sB(m,2)$, obtained by changing~$\bb_1$ to~$\bb_2$, and we denote by $\mu'\colon\so{[0,m+1]}\to\sB'(m,1)$ the corresponding polarized measure. In particular, $\mu'(i-1,i)=\ba_i$ for $1\leq i\leq m$, $\mu'(m,m+1)=\bb_2$, and $\mu'(0,m+1)=\bp$.

Now we define a map $\nu\colon\so{\Lambda}\to\sB(m,2)$ as follows:
 \begin{align*}
 \nu(i,j)&=\mu(i,j)\,,&&\text{for }1\leq i<j\leq m+1\,,\\
 \nu(i,j)&=\mu'(-j,-i)\,,&&\text{for }-m-1\leq i<j\leq -1\,,\\
 \nu(-i,j)&=\mu(0,i\wedge j)=\mu'(0,i\wedge j)\,,&&\text{for }i,j\in[m+1]\,. 
 \end{align*}
We claim that~$\nu$ is a $U$-polarized measure on~$\Lambda$. Let $x<y<z$ in~$\Lambda$, we need to prove that $\nu(x,z)\leq\nu(x,y)\vee\nu(y,z)$, while $y\in U$ implies that $\nu(x,y)\leq\nu(x,z)$ and $y\notin U$ implies that $\nu(y,z)\leq\nu(x,z)$.

If either $z<0$ or $x>0$, then the result follows from~$\mu'$ and $\mu$ being 
$U$-polarized measures. Now assume that $x<0$ and $z>0$ and set $x'=-x$. Then $\nu(x,z)=\mu(0,x'\wedge z)$. If $y\in U$, then $\nu(x,y)=\mu(0,x'\wedge y)\leq\mu(0,x'\wedge z)=\nu(x,z)$. Further, $x'\leq y$ implies that $x'\wedge y=x'\wedge z=x'$, thus $\mu(x'\wedge y,x'\wedge z)=0\leq\mu(y,z)$; while $y\leq x'$ implies that $\mu(x'\wedge y,x'\wedge z)=\mu(y,x'\wedge z)\leq\mu(y,z)$ (because~$\mu$ is a polarized measure). In each case, $\mu(x'\wedge y,x'\wedge z)\leq\mu(y,z)$, so
 \[
 \nu(x,z)=\mu(0,x'\wedge z)\leq
 \mu(0,x'\wedge y)\vee\mu(x'\wedge y,x'\wedge z)\leq\nu(x,y)\vee\nu(y,z)\,.
 \]

If $y\notin U$, then the element $y'=-y$ belongs to~$U$ and
$y'<x'$. Further, $\nu(y,z)=\mu'(0,y'\wedge z)\leq\mu'(0,x'\wedge
z)=\nu(x,z)$. As above, $\mu'(y'\wedge z,x'\wedge z)\leq\mu'(y',x')$, so we obtain
 \[
 \nu(x,z)=\mu'(0,x'\wedge z)\leq
 \mu'(0,y'\wedge z)\vee\mu'(y'\wedge z,x'\wedge z)\leq\nu(y,z)\vee\nu(x,y)\,.
 \]
This completes the proof of~$\nu$ being a $U$-polarized measure.

Since $\nu(-m-1,m+1)=\bp$, $\nu(-1,1)=\ba_1$, $\nu(m,m+1)=\bb_1$, $\nu(-m-1,-m)=\bb_2$, and $\nu(i,i+1)=\ba_{i+1}$ for $1\leq i\leq m-1$, the range of~$\nu$ generates~$\sB(m,2)$ as a \jzs. Now it remains to verify the V-condition. In order to do this, it suffices to prove that for every $(x,y)\in\so{\Lambda}$, every positive integer~$n$, and every minimal join-covering in~$\sB(m,2)$ of the form $\nu(x,y)\leq\bigvee_{1\leq j\leq n}\bc_j$ (\emph{observe that by minimality, all~$\bc_j$ are \jirr}), there are a positive integer~$k$ and a subdivision $x=z_0<z_1<\cdots<z_k=y$ in~$\Lambda$ such that each $\nu(z_i,z_{i+1})$ is contained in some~$\bc_j$. Since~$\bp$ is the only \jirr\ element of~$\sB(m,2)$ that is not join-prime, it suffices to solve this problem in each case $\nu(x,y)=\bigvee_{1\leq j\leq n}\bc_j$ with $n\geq2$, and $\nu(x,y)=\bp<\bigvee_{1\leq j\leq n}\bc_j$.

We begin with the first case. Since all the~$\bc_j$ belong to $A\cup\set{\bb_1}$ if $x>0$ and to $A\cup\set{\bb_2}$ if $y<0$, our refinement problem has a solution if either $x>0$ or $y<0$ (because~$\mu$ and~$\mu'$ are V-measures to~$\sB(m,1)$ and~$\sB'(m,1)$, respectively). Suppose now that $x<0$ and $y>0$; set $x'=-x$. Then $\mu(0,x'\wedge y)=\nu(x,y)=\bigvee_{1\leq j\leq n}\bc_j$. We separate cases. If $x'\leq y$, then $\mu(0,x')=\bigvee_{1\leq j\leq n}\bc_j$ is a minimal join-covering with $n\geq2$, thus $x'\leq m$ (\emph{this is because $\mu(0,m+1)=\bp$}) and our join-covering is equivalent, up to permutation, to $\nu(x,y)=\bigvee_{1\leq j\leq x'}\ba_j$, for which a refinement is given by the subdivision $x<x+1<\cdots<-1<y$, with successive measures $\nu(x,x+1)=\ba_{x'}$, $\nu(x+1,x+2)=\ba_{x'-1}$, \dots, $\nu(-2,-1)=\ba_2$, $\nu(-1,y)=\ba_1$. If $x'\geq y$, then $\mu(0,y)=\bigvee_{1\leq j\leq n}\bc_j$ is a minimal join-covering with $n\geq2$, thus $y\leq m$ and our join-covering is equivalent, up to permutation, to $\nu(x,y)=\bigvee_{1\leq j\leq y}\ba_j$, for which a refinement is given by the subdivision $x<1<\cdots<y-1<y$, with successive measures $\nu(x,1)=\ba_1$, $\nu(1,2)=\ba_2$, \dots, $\nu(y-1,y)=\ba_y$.

It remains to deal with the minimal join-coverings of the form $\nu(x,y)=\bp<\bigvee_{1\leq j\leq n}\bc_j$. Necessarily, $x=-m-1$, $y=m+1$, and our covering is equivalent, up to permutation, to a covering of the form 
 \[
 \nu(-m-1,m+1)=\bp<\ba_1\vee\ba_2\vee\cdots\vee\ba_m\vee\bb_l\,,
 \quad\text{for some }l\in\set{1,2}\,.
 \]
If $l=1$, then a refinement is given by $-m-1<1<2<\cdots<m<m+1$, with successive measures~$\ba_1$, $\ba_2$, \dots, $\ba_m$, $\bb_1$. If $l=2$, then a refinement is given by $-m-1<-m<-m+1<\cdots<-1<m+1$, with successive measures~$\bb_2$, $\ba_m$, $\ba_{m-1}$, \dots, $\ba_1$.
\end{proof}

In particular, it follows from Proposition~\ref{P:EmbBm2}
that~$\sB(2,2)$ has a lattice embedding into~$\sA_{\set{4,5}}(6)$,
thus into~$\sP(6)$. It can be shown that~$\sB(2,2)$ has no lattice embedding into~$\sP(n)$, for $n \leq 5$.

\section{A lattice that cannot be embedded into any permutohedron}\label{S:B33}

The main goal of the present section is to provide a proof of the following result, which implies that \emph{not every finite bounded lattice can be embedded into a permutohedron}.

\begin{theorem}\label{T:B33}
The lattice~$\sB(3,3)$ cannot be embedded into any permutohedron.
\end{theorem}

In order to prove Theorem~\ref{T:B33}, we denote, as in earlier sections, the \jirr\ elements of~$\sB(3,3)$ by~$\ba_1$, $\ba_2$, $\ba_3$, $\bb_1$, $\bb_2$, $\bb_3$, and~$\bp$, with $\ba_i<\bp$ for each $i\in\set{1,2,3}$. We also set $\ba=\ba_1\vee\ba_2\vee\ba_3$. We suppose that there exists a lattice embedding $\varphi\colon\sB(3,3)\hookrightarrow\sP(\ell)$ for some positive integer~$\ell$. Now~$\sP(\ell)$ is a subdirect product of its associated Cambrian lattices~$\sA_U(\ell)$ (cf. Proposition~\ref{P:DecompPU(n)}), thus, since~$\sB(3,3)$ is subdirectly irreducible (cf. Jipsen and Rose \cite[Theorem~4.17]{JiRo}), there is a lattice embedding $\psi\colon\sB(3,3)\hookrightarrow\sA_U(\ell)$ for some $U\subseteq[\ell]$. Now we define a new lattice~$K$ by setting
 \[
 K=\begin{cases}
 \sB(3,3)\,,&\quad\text{if }\psi(1_{\sB(3,3)})=1_{\sA_U(n)}\,,\\
 \sB(3,3)\cup\set{\infty}\,,&\quad\text{otherwise},
 \end{cases}
 \]
and we extend~$\psi$ to~$K$ by setting $\psi(\infty)=1_{\sA_U(n)}$ (in case $\psi(1_{\sB(3,3)})\neq1_{\sA_U(n)}$). Now~$\psi$ is an unit-preserving lattice embedding from~$K$ into~$\sA_U(\ell)$. By Proposition~\ref{P:betwmuphi}, the range of the dual $U$-polarized measure $\mu\colon\cI_{\ell}\to K$ generates~$K$ as a \jzs.

In particular, $\bp$ is a join of elements in the range of~$\mu$. Since~$\bp$ is \jirr, it follows that there exists $(x,y)\in\cI_{\ell}$ such that $\bp=\mu(x,y)$. Pick such an $(x,y)$ with $y-x$ minimal. For each $i\in[3]$, we say that a subdivision $x=z_0<z_1<\cdots<z_{n}=y$ is \emph{subordinate to~$\bb_i$} if
 \begin{equation}\label{Eq:SubdivzijB33}
 \text{either }\mu(z_j,z_{j+1})\leq\bb_i\text{ or }
 \mu(z_j,z_{j+1})\leq\ba_l\text{ for some }l\in[3]\,,
 \text{ for each }j<n\,.
 \end{equation}
Since~$\mu$ is a V-measure and $\mu(x,y)=\bp\leq\ba_1\vee\ba_2\vee\ba_3\vee\bb_i$, there exists certainly such a subdivision.
Observe that as $\bp\leq\ba_1\vee\ba_2\vee\ba_3\vee\bb_i$ is a minimal covering, each element of $\set{\ba_1,\ba_2,\ba_3,\bb_i}$ appears at least once among the elements $\mu(z_j,z_{j+1})$. In particular, $n\geq4$.

Recall that~$U^\cpl$ denotes the complement of~$U$. Say that a \emph{peak index} of a subdivision $x=z_0<z_1<\cdots<z_n=y$ is an index $j\in[0,n-1]$ such that $z_j\in U\cup\set{x}$ and $z_{j+1}\in U^\cpl\cup\set{y}$. We shall call the pair $(z_j,z_{j+1})$ the \emph{peak associated to~$j$}.

\begin{lemma}\label{L:ExistsPeak}
Let $i\in[3]$.
Each subdivision $x=z_0<z_1<\cdots<z_{n}=y$ subordinate to~$\bb_i$ has a peak index. Furthermore, $\mu(x,z_j)\leq\ba$ and $\mu(z_{j+1},y)\leq\ba$ while $\mu(z_j,z_{j+1})\leq\bb_i$, for each peak index~$j$.
\end{lemma}

\begin{note}
In the statement above, we are using again the convention $\mu(z,z)=0$ for each $z\in[\ell]$.
\end{note}

\begin{proof}
If $z_j\in U\cup\set{x}$ for some $j\in[0,n-1]$, then, taking the largest such~$j$, we obtain that $z_{j+1}\in U^\cpl\cup\set{y}$. On the other hand, if $z_{j+1}\in U^\cpl\cup\set{y}$ for some $j\in[0,n-1]$, then, taking the least such~$j$, we obtain that $z_j\in U\cup\set{x}$. In both cases, $j$ is a peak index; thus such an index always exists.

Let~$j$ be a peak index. {}From~$\mu$ being a $U$-polarized measure it follows that $\mu(x,z_j)\leq\bp$ and $\mu(z_{j+1},y)\leq\bp$. Therefore, by the minimality assumption on $y-x$, it follows that $\mu(x,z_j)\leq\ba$ and $\mu(z_{j+1},y)\leq\ba$, hence
 \[
 \bp=\mu(x,y)\leq\mu(x,z_j)\vee\mu(z_j,z_{j+1})\vee\mu(z_{j+1},y)
 \leq\ba\vee\mu(z_j,z_{j+1})\,.
 \]
Since~$\bp\nleq\ba$, it follows that $\mu(z_j,z_{j+1})\nleq\ba$, thus, by~\eqref{Eq:SubdivzijB33}, $\mu(z_j,z_{j+1})\leq\bb_i$.
\end{proof}

Say that a subdivision subordinate to~$\bb_i$ is \emph{normal} if it has a peak index~$j$ such that for each $k\in[0,n-1]\setminus\set{j}$ there exists $l\in[3]$ such that $\mu(z_k,z_{k+1})\leq\ba_l$.

\begin{lemma}\label{L:NormSubd}
There exists a normal subdivision subordinate to~$\bb_i$, for each index $i\in[3]$. Furthermore, for each such subdivision $x=z_0<z_1<\cdots<z_{n}$ and each $k\in[n-1]$, $k\leq j$ implies that $z_k\in U$ while $j+1\leq k$ implies that $z_k\in U^\cpl$.
\end{lemma}

\begin{note}
This implies, of course, that the peak index of a normal subdivision is unique.
\end{note}

\begin{proof}
By Lemma~\ref{L:ExistsPeak}, every subdivision $x=z_0<z_1<\cdots<z_{n}=y$ subordinate to~$\bb_i$ has a peak index~$j$, while $\mu(x,z_j)\leq\ba$ and $\mu(z_{j+1},y)\leq\ba$. Since $\ba=\ba_1\vee\ba_2\vee\ba_3$ and as~$\mu$ is a V-measure, there are natural numbers~$p$, $q$ and decompositions $x=s_0<s_1<\cdots<s_p=z_j$ and $z_{j+1}=s_{p+1}<s_{p+2}<\cdots<s_{p+q+1}=y$ such that for each $k\in[0,p+q]\setminus\set{p}$ there exists $l\in[3]$ such that $\mu(s_k,s_{k+1})\leq\ba_l$. Obviously, the subdivision
 \[
 x=s_0<s_1<\cdots<s_p<s_{p+1}<\cdots<s_{p+q+1}=y
 \]
is normal, with~$p$ as a peak index.

Now let $x=z_0<z_1<\cdots<z_{n}=y$ be a normal subdivision subordinate to~$\bb_i$, with peak index~$j$, and let $k\in[n-1]$. Suppose first that $k\leq j$. If $z_k\in U^\cpl$, then, as~$\mu$ is a $U$-polarized measure, $\mu(z_k,y)\leq\mu(x,y)=\bp$, thus, by the minimality assumption on~$y-x$, $\mu(z_k,y)\leq\ba$. However, from $\mu(z_l,z_{l+1})\leq\ba$ for each $l<k$ it follows that $\mu(x,z_k)\leq\ba$, thus
 \[
 \bp=\mu(x,y)\leq\mu(x,z_k)\vee\mu(z_k,y)\leq\ba\,,
 \]
\contr. It follows that $z_k\in U$. Likewise, $j+1\leq k$ implies that $z_k\notin U$.
\end{proof}

Now Lemma~\ref{L:NormSubd} ensures that for each $i\in[3]$, there exists a normal subdivision $x=z^i_0<z^i_1<\cdots<z^i_{n_i}=y$ subordinate to~$b_i$. Set $Z_i=\setm{z^i_j}{0\leq j\leq n_i}$ and denote by $(s_i,t_i)$ the unique peak of that subdivision; so $x\leq s_i<t_i\leq y$.

\begin{lemma}\label{L:Comptitj}
Let $i,j\in[3]$ be distinct. If $t_i\leq t_j$, then $\mu(t_j,y)\leq\ba_l$ for some $l\in[3]$.
\end{lemma}

\begin{proof}
If $t_j=y$, then the conclusion holds trivially. Thus suppose that $t_j<y$. Since $(s_j,t_j)$ is a peak, $t_j\in U^\cpl$. Moreover, $t_i<y$, thus $t_i\leq z^i_{n_i-1}<y$. Since $(s_i,t_i)$ is a peak and the subdivision associated to~$Z_i$ is normal, it follows that $\mu(z^i_{n_i-1},y)\leq\ba_l$ for some $l\in[3]$.

We claim that $z^i_{n_i-1}<t_j$. Suppose otherwise, that is, $t_j\leq z^i_{n_i-1}$. Since~$x$ and~$z^i_{n_i-1}$ both belong to~$Z_i$, the inequality $\mu(x,z^i_{n_i-1})\leq\ba\vee\bb_i$ holds. Now $t_i=z^i_m$ for some $m\in[n_i-1]$, thus, as the subdivision associated to~$Z_i$ is normal,
 \[
 \mu(t_i,z^i_{n_i-1})\leq\bigvee_{m\leq k<n_i-1}\mu(z^i_k,z^i_{k+1})\leq\ba\,.
 \]
It follows that
 \begin{align*}
 \mu(t_j,z^i_{n_i-1})&\leq\mu(t_i,z^i_{n_i-1})&&
 (\text{because }t_i\leq t_j\leq z^i_{n_i-1}\text{ and }t_j\in U^\cpl)\\
 &\leq\ba\,.
 \end{align*}
Since $\mu(x,t_j)\leq\ba\vee\bb_j$ (because~$x$ and~$t_j$ both belong to~$Z_j$), it follows that $\mu(x,z^i_{n_i-1})\leq\mu(x,t_j)\vee\mu(t_j,z^i_{n_i-1})\leq\ba\vee\bb_j$. Therefore, $\mu(x,z^i_{n_i-1})\leq(\ba\vee\bb_i)\wedge(\ba\vee\bb_j)=\bp$, thus, by the minimality statement on $y-x$,  we get $\mu(x,z^i_{n_i-1})\leq\ba$. Since $\mu(z^i_{n_i-1},y)\leq\ba_l$, it follows that $\bp=\mu(x,y)\leq\mu(x,z^i_{n_i-1})\vee\mu(z^i_{n_i-1},y)\leq\ba$, \contr.

By the claim above, $z^i_{n_i-1}<t_j$. Since $t_j\in U^\cpl$, it follows that $\mu(t_j,y)\leq\mu(z^i_{n_i-1},y)\leq\ba_l$.
\end{proof}

The following dual version of Lemma~\ref{L:Comptitj} can be proved likewise.

\begin{lemma}\label{L:Comptitjdual}
Let $i,j\in[3]$ be distinct. If $s_i\leq s_j$, then $\mu(x,s_i)\leq\ba_k$ for some $k\in[3]$.
\end{lemma}

Now we can conclude the proof of Theorem~\ref{T:B33}. We may assume without loss of generality that $t_1\leq t_2\leq t_3$. It follows from Lemma~\ref{L:Comptitj} that $\mu(t_2,y)\leq\ba_l$ for some $l\in[3]$. Since $t_2\leq t_3\leq y$ and $t_3\in U^\cpl\cup\set{y}$, it follows that $\mu(t_3,y)\leq\mu(t_2,y)\leq\ba_l$.

Next, suppose that $s_2\leq s_3$. It follows from Lemma~\ref{L:Comptitjdual} that $\mu(x,s_2)\leq\ba_k$ for some $k\in[3]$, so
 \[
 \bp=\mu(x,y)\leq\mu(x,s_2)\vee\mu(s_2,t_2)\vee\mu(t_2,y)
 \leq\ba_k\vee\ba_l\vee\bb_2\,,
 \]
\contr. On the other hand, if $s_3\leq s_2$, then, again by Lemma~\ref{L:Comptitjdual}, $\mu(x,s_3)\leq\ba_k$ for some $k\in[3]$, so
 \[
 \bp=\mu(x,y)\leq\mu(x,s_3)\vee\mu(s_3,t_3)\vee\mu(t_3,y)
 \leq\ba_k\vee\ba_l\vee\bb_3\,,
 \]
\contr\ again. This completes the proof of Theorem~\ref{T:B33}.

\bigskip
By combining the result of Theorem~\ref{T:B33} with those of Proposition~\ref{P:Perm(n)compl}, Lemma~\ref{L:DualBmn}, and Proposition~\ref{P:EmbBm2}, we obtain the following analogue, for permutohedra, of Theorem~\ref{T:EmbBn01}.

\begin{theorem}\label{T:EmbBn012}
Let~$m$ and~$n$ be natural numbers. Then the lattice~$\sB(m,n)$ embeds into some permutohedron if{f} either $m\leq2$ or~$n\leq2$.
\end{theorem}

\section{A large permutohedron with a preimage of $\sB(3,3)$}\label{S:SplitEq}

After several unsuccessful attempts to turn Theorem~\ref{T:B33} to an identity holding in all permutohedra while failing in~$\sB(3,3)$, we (the authors of the present paper) started wondering whether it could actually be the case that~$\sB(3,3)$ satisfies every lattice-theoretical identity satisfied by all permutohedra! The goal of the present section is to provide a proof that this guess was correct.

In order to do this, we shall need the notion of \emph{splitting identity} of a finite, bounded, subdirectly irreducible lattice. Such lattices are often called \emph{splitting lattices} (after McKenzie~\cite{McKe72}, see also Freese, Je\v{z}ek, and Nation~\cite{FJN}). It is a classical result of lattice theory (cf. Freese, Je\v{z}ek, and Nation~\cite[Corollary~2.76]{FJN}) that for every splitting lattice~$K$, there exists a largest lattice variety~$\cC_K$ which is maximal with respect to not containing~$K$ as a member. Furthermore, $\cC_K$ can be defined by a single lattice identity, called a \emph{splitting identity} for~$K$, and there is an effective way to compute such an identity.

We shall apply this algorithm (given by \cite[Corollary~2.76]{FJN}) to the six-element set $X=\set{\sx_1,\sx_2,\sx_3,\sy_1,\sy_2,\sy_3}$, the lattice~$\sB(3,3)$, with $u=\bp$ and $v=\ba=\ba_1\vee\ba_2\vee\ba_3$, and the unique lattice homomorphism $f\colon\FL(X)\twoheadrightarrow\sB(3,3)$ such that $f(\sx_i)=\ba_i$ and $f(\sy_i)=\bb_i$ for each $i\in[3]$ (where $\FL(X)$ denotes the free lattice on~$X$). {}From $\bp=\bigwedge_{j\in\set{1,2}}(\ba_1\vee\ba_2\vee\ba_3\vee\bb_j)$ it follows that~$f$ is surjective.

For each $i\in[3]$, denote by~$i'$ and~$i''$ the other two elements of~$[3]$. We introduce new lattice terms by
 \begin{gather*}
 \sx=\sx_1\vee\sx_2\vee\sx_3\,,\quad\sy=\sy_1\vee\sy_2\vee\sy_3\,,\\
 \hsx_i=\sx_{i'}\vee\sx_{i''}\vee\sy\,,\quad
 \hsy_i=\sy_{i'}\vee\sy_{i''}\vee\sx\,,\quad\text{for each }i\in[3]\,,
 \end{gather*}
and the corresponding elements of~$\sB(3,3)$,
 \begin{gather*}
 \ba=\ba_1\vee\ba_2\vee\ba_3\,,\quad\bb=\bb_1\vee\bb_2\vee\bb_3\,,\\
 \hba_i=\ba_{i'}\vee\ba_{i''}\vee\bb\,,\quad
 \hbb_i=\bb_{i'}\vee\bb_{i''}\vee\ba\,,\quad\text{for each }i\in[3]\,.
 \end{gather*}
The $0$th stage~$\beta_0$ of the lower limit table (cf. Freese, Je\v{z}ek, and Nation \cite[Theorem~2.4]{FJN}) on the \jirr\ elements of~$\sB(3,3)$ is given by
 \[
 \beta_0(\ba_i)=\sx_i\text{ and }\beta_0(\bb_i)=\sy_i\quad
 \text{for each }i\in[3]\,,\quad\beta_0(\bp)=1\,.
 \]
Then, using the only minimal join-coverings of~$\sB(3,3)$, namely $\bp<\ba_1\vee\ba_2\vee\ba_3\vee\bb_j$ for each $j\in[3]$, we obtain the first stage~$\beta_1$ of the lower limit table of $\sB(3,3)$ on the \jirr\ elements of~$\sB(3,3)$:
 \begin{align*}
 \beta_1(\ba_i)&=\beta_0(\ba_i)=\sx_i\,,\\
 \beta_1(\bb_i)&=\beta_0(\bb_i)=\sy_i\,,\\
 \intertext{while}
 \beta_1(\bp)&=\bigwedge_{j=1}^3
 \bigl(\beta_0(\ba_1)\vee\beta_0(\ba_2)\vee\beta_0(\ba_3)
 \vee\beta_0(\bb_j)\bigr)\\
 &=\bigwedge_{j=1}^3\bigl( \sx_1\vee\sx_2\vee\sx_3\vee\sy_j\bigr)\,.
 \end{align*}
Since $\rD_1(\sB(3,3))=\sB(3,3)$, it follows from \cite[Lemma~2.7]{FJN} that $\beta=\beta_1$.

Similar calculations yield the upper limit table for~$\sB(3,3)$ on the \mirr\ elements of~$\sB(3,3)$:
 \begin{gather*}
 \alpha_0(\hba_i)=\hsx_i\,,\quad\alpha_0(\hbb_i)=\hsy_i,,\quad
 \alpha_0(\ba)=\sx\,,\\
 \alpha_1(\hba_i)=\alpha_0(\hba_i)=\hsx_i\,,\quad
 \alpha_1(\hbb_i)=\alpha_0(\hbb_i)=\hsy_i\,,\\
 \alpha_1(\ba)=\sx\vee\bigvee_{i=1}^3\bigl(
 \hsx_i\wedge\hsy_1\wedge\hsy_2\wedge\hsy_3\bigr)\,.
 \end{gather*}
Furthermore, as obviously
 \[
 \sx_{i'}\vee\sx_{i''}\leq\hsx_i\wedge\hsy_1\wedge\hsy_2\wedge\hsy_3\,,
 \]
we obtain $\sx\leq\bigvee_{i=1}^3(\hsx_i\wedge\hsy_1\wedge\hsy_2\wedge\hsy_3)$, thus
 \[
 \alpha_1(\ba)=\bigvee_{i=1}^3\bigl(
 \hsx_i\wedge\hsy_1\wedge\hsy_2\wedge\hsy_3\bigr)\,.
 \]
Since $\rD_1(\sB(3,3)^\op)=\sB(3,3)^\op$, it follows that $\alpha=\alpha_1$.

Consequently, by Freese, Je\v{z}ek, and Nation~\cite[Corollary~2.76]{FJN}, a splitting identity for~$\sB(3,3)$ is given by
 \begin{equation}\label{Eq:Spl1B330}
 \bigwedge_{1\leq j\leq 3}(\sx_1\vee\sx_2\vee\sx_3\vee\sy_j)\leq
 \bigvee_{1\leq i\leq 3}(\hsx_i\wedge
 \hsy_1\wedge\hsy_2\wedge\hsy_3)\,.
 \end{equation}
While all the splitting identities for~$\sB(3,3)$ are equivalent, we shall work with the one given by~\eqref{Eq:Spl1B330}. We obtained the example underlying Theorem~\ref{T:Perm12} with the assistance of the \texttt{Mace4} component of the \texttt{Prover9 - Mace4} software, see McCune~\cite{McCune}.

\begin{theorem}\label{T:Perm12}
Set $U=\set{5,6,9,10,11}$. Then the Cambrian lattice $\sA_U(12)$ does not satisfy the identity~\textup{\eqref{Eq:Spl1B330}}. Consequently, $\sB(3,3)$ is the homomorphic image of a sublattice of~$\sA_U(12)$.
\end{theorem}

\begin{proof}
We consider the elements $\ba_1$, $\ba_2$, $\ba_3$, $\bb_1$, $\bb_2$, and~$\bb_3$ of~$\sA_U(12)$ defined as
 \begin{align*}
 \ba_1&=
 \pUji{1}{5}{U}\vee\pUji{2}{3}{U}\vee\pUji{8}{12}{U}\vee\pUji{10}{11}{U}\,;\\
 \ba_2&=\pUji{3}{4}{U}\vee\pUji{5}{9}{U}\,;\\
 \ba_3&=\pUji{4}{8}{U}\vee\pUji{9}{10}{U}\,;\\
 \bb_1&=\pUji{1}{2}{U}\,;\\
 \bb_2&=\pUji{6}{7}{U}\,;\\
 \bb_3&=\pUji{11}{12}{U}\,.
 \end{align*}
Due to the subdivisions
 \begin{align*}
 &1<2<3<4<8<12\,,&&\text{with successive measures }
 \bb_1\,,\ \ba_1\,,\ \ba_2\,,\ \ba_3\,,\ \ba_1\,,\\
 &1<5<6<7<8<12\,,&&\text{with successive measures }
 \ba_1\,,\ \ba_2\,,\ \bb_2\,,\ \ba_3\,,\ \ba_1\,,\\
 &1<5<9<10<11<12\,,&&\text{with successive measures }
 \ba_1\,,\ \ba_2\,,\ \ba_3\,,\ \ba_1\,,\ \bb_3\,,
 \end{align*}
we obtain that the pair $(1,12)$ belongs to $\bigwedge_{j=1}^3(\ba_1\vee\ba_2\vee\ba_3\vee\bb_j)$. On the other hand, evaluating the two sides of~\eqref{Eq:Spl1B330} at the~$\ba_i$s and~$\bb_i$s yields that $(1,12)$ does not belong to the right hand side of the equation. Therefore, $\sA_U(12)$ does not satisfy~\eqref{Eq:Spl1B330}.

Since~\eqref{Eq:Spl1B330} is a splitting identity for~$\sB(3,3)$, it follows that~$\sB(3,3)$ belongs to the lattice variety generated by~$\sA_U(12)$. Since~$\sA_U(12)$ is subdirectly irreducible (cf. Proposition~\ref{P:DecompPU(n)}), the final statement of Theorem~\ref{T:Perm12} follows from J\'onsson's Lemma (cf. Corollary~1.5 and Lemma~1.6 in Jipsen and Rose \cite{JiRo}).
\end{proof}

\begin{corollary}\label{C:Perm12}
The lattice~$\sB(3,3)$ satisfies every lattice-theoretical identity satisfied by~$\sA_U(12)$, thus also every lattice-theoretical identity satisfied by the permutohedron~$\sP(12)$. In particular, $\sB(3,3)$ satisfies every lattice-theoretical identity satisfied by every permutohedron.
\end{corollary}

\section{Open problems}\label{S:Pbs}

Almost every nontrivial question about embedding finite lattices into
Tamari lattices, permutohedra, or related objects, is open, so we shall just list a few here. Examples of fundamental questions are the following:
\begin{itemize}
\item[(1)] Is it decidable whether a given finite lattice embeds into some permutohedron (resp., Tamari lattice)?

\item[(2)] Is it decidable whether a given lattice-theoretical identity holds in all permutohedra (resp., Tamari lattices)?

\item[(3)] Can the lattice variety generated by all permutohedra (resp., Tamari lattices) be defined by a recursive set of lattice identities?

\item[(4)] Is the class of
all sublattices of Tamari lattices the intersection of a lattice variety
with the class of all finite bounded lattices? In particular, if a lattice~$L$ can be embedded into some Tamari lattice, is this also the case for all homomorphic images of~$L$? (By Theorems~\ref{T:B33}
and~\ref{T:Perm12}, the analogue of this problem for permutohedra has a negative answer.)

\item[(5)] Does there exist a nontrivial lattice-theoretical identity satisfied by all permutohedra? (The results of Section~\ref{S:B33} suggest a negative answer, while the results of Section~\ref{S:SplitEq} suggest a positive answer.)

  \item[(6)] Does every closed interval of a Tamari lattice (resp., a permutohedron) have a $(0,1)$-preserving lattice embedding into some Tamari lattice (resp., permutohedron)?

\end{itemize}

Caspard, Le Conte de Poly-Barbut, and Morvan proved in~\cite{CLM04}
that every finite Coxeter lattice (i.e., weak Bruhat order on a finite Coxeter group) is bounded.  All the analogues for Coxeter lattices of the questions above are open as well. Can every finite Coxeter lattice
be embedded into some permutohedron? (This is the case for Coxeter
lattices of type~$B_n$, but it needs to be worked out for other types,
such as~$D_n$.)

\end{document}